\documentclass[11pt,a4paper]{amsart}
\usepackage[initials]{amsrefs}
\usepackage{amssymb}
\usepackage[english]{babel}
\usepackage{enumerate}
\usepackage[T1]{fontenc}
\usepackage{graphicx}
\usepackage[utf8]{inputenc}
\usepackage{mathrsfs}
\usepackage{mathtools}
\usepackage{microtype}
\usepackage{multirow}
\usepackage{siunitx}
\usepackage{subcaption}
\usepackage{tikz}

\usepackage[hidelinks]{hyperref}
\usepackage[nameinlink]{cleveref}
\usepackage[initials]{amsrefs}

\newtheorem{theorem}{Theorem}[section]
\newtheorem{proposition}[theorem]{Proposition}
\newtheorem{lemma}[theorem]{Lemma}

\theoremstyle{definition}
\newtheorem{definition}[theorem]{Definition}

\theoremstyle{remark}
\newtheorem{remark}[theorem]{Remark}

\numberwithin{equation}{section}

\newcommand{\map}[3]{ #1 \colon #2 \to #3 }
\newcommand{\restr}[2]{ \left.#1\right\rvert_{#2}}
\newcommand{\ad}[2]{% \repeat already defined
    #1^{\foreach \n in {1,...,#2}{\sharp}}
}

\newcommand{\R}[0]{ \mathbb{R} }

\newcommand{\Z}[0]{ \mathbb{Z} }
\newcommand{\N}[0]{ \mathbb{N} }
\newcommand{\C}[0]{ \mathbb{C} }

\DeclarePairedDelimiter{\abs}{\lvert}{\rvert}

\DeclarePairedDelimiter{\jb}{\langle}{\rangle}

\DeclareMathOperator{\arccot}{arccot}
\DeclareMathOperator{\arcosh}{arcosh}
\DeclareMathOperator{\diag}{diag}
\DeclareMathOperator{\re}{Re}
\DeclareMathOperator{\sech}{sech}
\DeclareMathOperator{\sgn}{sgn}
\DeclareMathOperator{\supp}{supp}

\begin{document}

\title{Solitary Gravity-Capillary Water Waves with Point Vortices}
\author{Kristoffer Varholm}
\address{Department of Mathematical Sciences, Norwegian University of Science and Technology, 7491 Trondheim, Norway}
\email{kristoffer.varholm@math.ntnu.no}
\thanks{The author acknowledges the support of the project Nonlinear Water Waves by the Research Council of Norway (Grant No. 231668). The author would also like to thank the referee for several helpful suggestions that helped improve the manuscript.}


\begin{abstract}
We construct small-amplitude solitary traveling gravity-capillary water waves with a finite number of point vortices along a vertical line, on finite depth. This is done using a local bifurcation argument. The properties of the resulting waves are also examined: We find that they depend significantly on the position of the point vortices in the water column.
\end{abstract}

\maketitle

\section{Introduction}
    The steady water-wave problem concerns two-dimensional water waves propagating with constant velocity and without change of shape. Historically, the focus has mainly been on irrotational waves, which are waves where the vorticity\footnote{Informally, the vorticity describes (twice) the velocity at which an infinitesimal paddle wheel placed in the fluid will rotate.}
    \[
        \omega \coloneqq \nabla \times w = v_x - u_y
    \]
    of the velocity field $w=(u,v)$ is identically zero. One reason for this is Kelvin's circulation theorem \cites{Johnson1997,Lighthill1978}, which says that a flow which is initially irrotational will remain so for all time, as long as it is only affected by conservative body forces (e.g. gravity). Another reason is mathematical, as the velocity field can then be written as the gradient of a harmonic function; the velocity potential. This enables the use of powerful tools from complex- and harmonic analysis, and the problem can be reduced to one on the boundary in a number of different ways \cites{Babenko1987,Nekrasov1921}. An important class of such waves are the Stokes waves, which are periodic waves that rise and fall exactly once every minimal period. The Stokes conjecture on the nature of the so-called Stokes wave of greatest height fueled research on waves throughout the 20th century, and would not be fully resolved until 2004 (see the survey \cite{Toland1996} and \cite{Plotnikov2004}, which settled the convexity of this wave).
    
    More recently, however, there has been renewed interest in rotational waves. There are several situations where such waves are appropriate, as effects like wind, temperature or salinity gradients can all induce rotation \cite{Mei1984}. Rotational waves can be markedly different from irrotational waves: For instance, in rotational waves it is possible to have internal stagnation points and critical layers of closed streamlines known as cat's eye vortices \cite{Ehrnstrom2012}.
    
    The first result on rotational waves came surprisingly early, in the beginning of the 1800s with \cite{Gerstner1809} (for a more modern exposition, see \cite{Constantin2001}). There, Gerstner gave the first, and still the only known, explicit (nontrivial) gravity-wave solution to the Euler equations on infinite depth. Although significant because it is an exact solution, it is viewed as more of a mathematical curiosity, even today (see \cite[Chapter~4.3]{Constantin2011}). Much later, in \cite{Dubreil-Jacotin1934}, came the first existence result for small-amplitude waves with quite general vorticity distributions. A vorticity distribution is a function $\map{\gamma}{\R}{\R}$ such that
    \[
        \Delta \psi = \gamma(\psi),
    \]
    where $\psi$ is the relative stream function (which, unlike the velocity potential, is still available for rotational waves, but is not harmonic). A sufficient, but not necessary, condition for such a vorticity distribution to exist is that the wave has no stagnation points. Several improvements have been made to the existence result of Dubreil-Jacotin, but it was not until the pioneering article \cite{Constantin2004} that large waves were constructed, using global bifurcation theory. This article sparked mathematical research into rotational waves.
    
    The use of a semi-hodograph transform in \cite{Dubreil-Jacotin1934} and \cite{Constantin2004}, and the corresponding deep-water result in \cite{Hur2006}, means that the resulting waves cannot exhibit critical layers. Since then, small-amplitude waves with constant vorticity and a critical layer have been constructed in \cite{Wahlen2009}, and later in \cite{Constantin2011a} with a different approach that allows for waves with overhanging surface profiles (there is numerical evidence for the existence of such waves, e.g. \cite{Vanden-Broeck1996}, but this is still an open problem). A reasonable next step is that of waves with an affine vorticity distribution, whose existence was shown in \cites{Ehrnstrom2011,Ehrnstroem2015}. Spurred by the above results there has also been interest in studying the properties and dynamics of these waves below the surface \cites{Wahlen2009,Ehrnstrom2012}. This had been done for linear waves in \cite{Ehrnstrom2008}. Several other avenues have also been considered: We mention heterogeneous waves both with \cites{Henry2014,Walsh2014a} and without \cite{Walsh2009,Escher2011} surface tension, waves with discontinuous vorticity \cite{Constantin2011b}, a variational approach \cite{Burton2011} and Hamiltonian formulation with center manifold reduction \cite{Groves2008}. Existence of large amplitude waves with constant vorticity and a critical layer was established in \cite{Matioc2014}, in the presence of capillary effects. There is also a forthcoming result for pure gravity waves \cite{Constantin2014}, using an entirely different approach.
    
    Common for all the previously mentioned works on rotational waves is the feature that the vorticity is supported on the entire fluid domain (due to the assumption of the existence of a vorticity distribution). Recently, gravity-capillary waves with compactly supported vorticity were constructed in \cite{Shatah2013}, on infinite depth. This includes small- and large-amplitude periodic waves with a point vortex, and small-amplitude solitary waves with either a point vortex or vortex patch. By a point vortex we mean that the vorticity is given by  a \(\delta\)-function, while we use vortex patch to mean that the vorticity is locally integrable and compactly supported. The waves with a point vortex are the simplest form of waves with compactly supported vorticity, and are in a sense ``almost irrotational''.
    
    In this paper, which is based on \cite{Varholm2014}, we extend the existence result for solitary small-amplitude waves with a point vortex to finite depth, and also give both qualitative and quantitative properties for these waves. The main approach to showing existence follows that of \cite{Shatah2013}, but we also treat the natural generalization of waves with several point vortices along a vertical line, and show existence for all but exceptional configurations of vortices. Finally, by finding an explicit expression for the rotational part of the stream function, we give some explicit expressions for the small-amplitude periodic waves with a point vortex on infinite depth that were constructed in \cite{Shatah2013}.
    
    An outline of the paper is as follows: In \Cref{section:formulation} we formulate the problem, and in  \Cref{section:functional_analytic_setting} we give the functional-analytic setting for this formulation. Then, in \Cref{section:local_bifurcation} we prove existence of small solutions, and give some properties for these. \Cref{section:several_point_vortices} treats the extension to several point vortices. The final section, \Cref{section:infinite_depth}, contains the explicit expressions for periodic waves on infinite depth.
\section{Formulation}
    \label{section:formulation}
    Under the assumption of inviscid (absence of viscosity) and incompressible (constant fluid density) flow, the governing equations of motion are the so-called incompressible Euler equations. For describing water waves on the open sea, these are realistic assumptions \cites{Johnson1997,Lighthill1978}, and standard. We will further assume two-dimensional flow under the influence of gravity, where the Cartesian coordinates $(x,y)$ describe the horizontal and vertical direction, respectively. Then the equations read
    \begin{equation}
        \begin{aligned}
            \label{eq:incompressible_euler_full}
            w_t + (w \cdot \nabla)w &= -\nabla p - g e_2,& &\text{(Conservation of momentum)}\\
            \nabla \cdot w &= 0,& &\text{(Conservation of mass)}
        \end{aligned}
    \end{equation}
    where $w=(u,v)$ is the velocity of the fluid, $p$ is the pressure distribution and $-ge_2 = (0,-g)$ is the constant gravitational acceleration\footnote{The constant $g$ is approximately $\SI{9.8}{m/s^2}$, varying by less than $\SI{0.4}{\percent}$ on the Earth's suface (see \cite{Hirt2013}).}.

    For convenience we place, at time $t$, the flat bottom at
    \[
        \{(x,y) \in \R^2 : y = -h\}
    \]
    and the surface at
    \[
        \{(x,y) \in \R^2: y = \eta(x,t)\},
    \]
    where $\eta$ describes the deviation of the free boundary. We assume that $\eta(\cdot,t)$ is bounded, continuous and strictly bounded below by $-h$. It should be emphasized that, due to the free boundary assumption, the function $\eta$ is a priori unknown; determining it is part of the problem.
    
    In addition to \Cref{eq:incompressible_euler_full}, we require boundary conditions to match our domain. In order to model the bottom being impermeable, we will demand that
    \begin{align}
        \restr{v}{y=-h} &= 0 && \text{(Kinematic boundary condition at bottom)}, \notag
        \intertext{with which we mean that $v(x,-h,t)=0$ for all $x$ and $t$. Next, we impose the condition}
        u\eta_x + \eta_t &= v && \text{(Kinematic boundary condition at surface)} \notag
        \intertext{at the surface. This equation is what connects the free boundary to the fluid, and is equivalent to demanding that particles at the surface will remain there. We also require that}
        \restr{p}{y=\eta} &= -\alpha^2 \kappa(\eta), &&\text{(Dynamic boundary condition)} \label{eq:young_laplace}
    \end{align}
    where $\alpha^2 > 0$ describes the surface tension and $\kappa$ is the nonlinear differential operator
    \[
        \kappa(\eta)  \coloneqq \left(\frac{\eta_x}{\jb{\eta_x}}\right)' = \frac{\eta_{xx}}{\jb{\eta_x}^3},
    \]
    yielding the curvature of the surface. The symbol $\jb{\cdot}$ denotes the Japanese bracket defined through $x \mapsto (1+\abs{x}^2)^{1/2}$. \Cref{eq:young_laplace} is known to physicists as the Young--Laplace equation, and states that the pressure difference across a fluid interface (in this case water/air) is proportional to its curvature.
       
    Note that in the lower limit $\alpha^2 = 0$, the dynamic boundary condition in \Cref{eq:young_laplace} corresponds to the assumption of constant pressure on the surface, but we will require that $\alpha^2$ be strictly positive. The proof of \Cref{thm:existence_point_vortex}, for example, relies upon the assumption that $\alpha^2>0$.
    
    \subsection{The vorticity equation}
        By taking the curl of \Cref{eq:incompressible_euler_full}, one obtains after some simple calculations that
        \begin{equation}
            \label{eq:vorticity_transport_equation_full}
            \omega_t + \nabla \cdot (\omega w)=0,
        \end{equation}
        which states that the vorticity $\omega$ is transported by the vector field $w$. Due to this, it is natural to expect that if the vorticity consists of a point vortex at some time, then it will remain a point vortex at all future times, and be transported with the flow. It should be emphasized that, for now, this is not justified by \Cref{eq:vorticity_transport_equation_full}; the multiplication of $\omega$ with $w$ is not well defined, as $w$ will not be smooth at the point vortex. Thus, we will have need of a weaker form of the equation. We remind the reader of the fundamental solution of the Poisson equation.
        
        \begin{proposition}[Newtonian potential]
            \label{prop:delta_vorticity}
            The distribution $\Gamma \in L_\textnormal{loc}^2(\R^2)$ defined by
            \[
                \Gamma(x,y) \coloneqq \frac{1}{4\pi}\log(x^2 + y^2)
            \]
            satisfies
            \[
                \nabla^\perp \Gamma(x,y) \coloneqq (-\Gamma_y,\Gamma_x)(x,y) = \frac{1}{2\pi} \frac{(-y,x)}{x^2 + y^2}
            \]
            and
            \[
                \Delta\Gamma = \nabla \times \nabla^\perp\Gamma = \delta.
            \]
        \end{proposition}

         If $\omega$ is of the form
        \[
            \omega(t) = \delta_{(x_0(t),y_0(t))},
        \]
        then we deduce from \Cref{prop:delta_vorticity} that $w$ is of the form
        \[
            w(x,y,t) =\frac{1}{2\pi}\frac{(y_0(t)-y,x-x_0(t))}{(x-x_0(t))^2 + (y-y_0(t))^2} + \hat{w}(x,y,t),
        \]
        where $\hat{w}$ satisfies $\nabla \cdot \hat{w} = 0$ and $\nabla \times \hat{w} = 0$, and is therefore smooth in space (see the discussion before \Cref{eq:velocity_in_terms_of_stream_function}). As the first term, which we may think of as the part of $w$ generated by the point vortex, is singular, divergence free and odd around $(x_0(t),y_0(t))$, it is not unreasonable to think that the dynamics of the point vortex should depend only on $\hat{w}$. In other words, that the path $t \mapsto (x_0(t),y_0(t))$ along which the point vortex moves should satisfy
        \begin{equation}
            \label{eq:vorticity_transport_equation_weak}
            (\dot{x}_0,\dot{y}_0) = \hat{w}.
        \end{equation}
        
        This can indeed be made rigorous. In \cite[Theorems 4.1 and 4.2]{Marchioro1994} it is proved that if one considers initial data consisting of a vortex patch converging in the sense of distributions to a point vortex, then the weak solutions of the vorticity transport equation converge to a moving point vortex in an appropriate sense. Moreover, the position of this point vortex satisfies \Cref{eq:vorticity_transport_equation_weak}. Thus, we will allow for point vortices, as long as they are propagated in the fluid as in \Cref{eq:vorticity_transport_equation_weak}.
    
    \subsection{Traveling waves}
        \label{section:traveling_waves}
        We now assume that there are functions $\tilde{w},\tilde{p}, \tilde{\eta}$, depending only on space, and a constant velocity $c \in \R$ such that
        \begin{align*}
            w(x,y,t) &= \tilde{w}(x-ct,y),\\
            p(x,y,t) &= \tilde{p}(x-ct,y),\\
            \eta(x,t) &= \tilde{\eta}(x-ct)
        \end{align*}
        for all relevant $x,y$ and $t$. Positive and negative $c$ then correspond to waves moving in the positive and negative $x$-directions, respectively. In the new \emph{steady} variables $(\tilde{x}, \tilde{y})=(x-ct,y)$, after dropping the tildes, our equations read
        \begin{align}
            (w\cdot \nabla)w -cw_x &= -\nabla p - g e_2, &&\text{(Conservation of momentum)} \label{eq:incompressible_euler_steady_momentum}\\
            \nabla \cdot w &= 0, &&\text{(Conservation of mass)} \notag
        \end{align}
        with boundary conditions
        \begin{align}
            v &=0,    &&\text{at $y=-h$}, && \text{(Kinematic)}\label{eq:boundary_condition_steady_bottom}\\
            (u-c)\eta'&=v,    &&\multirow{2}{*}{at $y = \eta(x)$,} && \text{(Kinematic)}\label{eq:boundary_condition_steady_kinematic}\\
            p&=-\alpha^2 \kappa(\eta),    && && \text{(Dynamic)}\label{eq:boundary_condition_steady_dynamic}
        \end{align}
        on the now time-independent domain
        \[
            \Omega(\eta) \coloneqq \{(x,y) \in \R^2 : -h < x < \eta(x)\}.
        \]
        We call the problem of finding $w, p$ and $\eta$ such that these equations are satisfied the \emph{steady water-wave problem}. Note also that the vorticity equation given in \Cref{eq:vorticity_transport_equation_weak} reduces to
        \begin{equation}
            \label{eq:vorticity_transport_equation_weak_steady}
            (c,0) = \hat{w}(x_0,y_0)
        \end{equation}
        for a point vortex centered at $(x_0,y_0) \in \Omega(\eta)$.
        
    \subsection{The Zakharov--Craig--Sulem formulation}
        \label{section:zakharov_craig_sulem}
        It turns out that it is possible to reduce the water-wave problem to an entirely one-dimensional one on the surface in a clever way. This is known as the Zakharov--Craig--Sulem formulation, and was first introduced by Zakharov in \cite{Zakharov1968}, and then later put on a firmer mathematical basis in \cites{Craig1992, Craig1993}. The original formulation relies on the fluid being irrotational, but it is in fact sufficient that this holds near the surface. This is where the compact support of the vorticity comes in.

        Suppose that we have solved the steady water wave problem for some $w, p, \eta$. It is then convenient to split the velocity $w$ as
        \[
            w = \hat{w} + W,
        \]
        where $\hat{w}$ is irrotational, that is, $\nabla \times \hat{w}=0$ and $\nabla \times W = \omega$. We also assume that both $\hat{w}$ and $W$ are divergence free. For us, $W$ will be known.
        
        Although we will allow for $\omega$ to be a singular distribution, the vector field $\hat{w}=(\hat{u},\hat{v})$ will be assumed to at least be in the Sobolev space $H^1(\Omega(\eta))^2$, and $W=(U,V)$ to be at least $L_\textnormal{loc}^1(\Omega(\eta))$. By the assumption of $\nabla \cdot \hat{w}=\nabla \cdot W = 0$, the differentials
        \[
            \hat{v}\,dx - \hat{u}\,dy, \quad V\,dx - U\,dy
        \]
        on $\Omega(\eta)$ are closed. Hence, as $\Omega(\eta)$ is simply connected, these differentials are exact by generalizations of the Poincaré lemma \cite[Theorems 2.1 and 3.1]{Mardare2008}. Thus, there are stream functions $\hat{\psi} \in H_{\textnormal{loc}}^2(\Omega(\eta)), \Psi \in L_\textnormal{loc}^1(\Omega(\eta))$ (by \cite[Corollary 2.1]{Deny1953--54}), determined uniquely modulo constants by $\hat{w}$ and $W$, such that
        \[
            \hat{w} = \nabla^\perp \hat{\psi}, \quad W = \nabla^\perp \Psi.
        \]
        Moreover, by the assumed curl of these vector fields, the function $\hat{\psi}$ is harmonic, and $\Psi$ satisfies $\Delta \Psi = \omega$. In particular, $\hat{\psi}$ is smooth, and so is $W$ outside the support of $\omega$.
        
        By the above, we thus have that
        \begin{equation}
            \label{eq:velocity_in_terms_of_stream_function}
            w = \nabla^\perp(\hat{\psi} + \Psi)
        \end{equation}
        holds for the velocity field $w$. Suppose now that $W$ and $\Psi$ are chosen such that $\Psi = 0$ at the bottom. Then the boundary condition in \Cref{eq:boundary_condition_steady_bottom} translates to $\hat{\psi}$ being constant along the bottom. Since $\hat{\psi}$ is unique modulo constants, we may as well take this condition to be
        \[
            \restr{\hat{\psi}}{y=-h} = 0
        \]
        instead.
        
        We will now apply the assumption that $\supp\omega \subseteq \Omega(\eta)$ is compact. This means that a velocity potential for $w$ exists on any simply connected domain not containing $\supp{\omega}$. \Cref{eq:incompressible_euler_steady_momentum} can then in turn be used to show that
        \[
            \nabla\left(-cu + \frac{1}{2}\abs{w}^2 + p + gy\right) = 0
        \]
        holds on any such domain. Hence, in particular, we obtain the surface Bernoulli equation
        \begin{equation}
            \label{eq:bernoulli}
            c(\hat{\psi}_y + \Psi_y) + \frac{1}{2}\abs{\nabla\hat{\psi} + \nabla \Psi}^2 -\alpha^2\kappa(\eta) + g\eta = C,\quad\text{at $y = \eta(x)$},
        \end{equation}
        where $C$ is a real constant. Here we have inserted the boundary condition for the pressure at the surface, \Cref{eq:boundary_condition_steady_dynamic}. We will set \(C=0\), since we are looking for localized waves. This can be seen by letting $|x| \to \infty$ in \Cref{eq:bernoulli}. We now need the following formal definitions to proceed, which will be specified later on.
        
        \begin{definition}[Harmonic extension operator]
            \label{def:harmonic_extension_operator}
            Given $\eta$, we define the harmonic extension operator $H(\eta)$ as the operator mapping a function $\map{\zeta}{\R}{\R}$ to the harmonic function $\map{\hat{\psi}}{\Omega(\eta)}{\R}$ satisfying
            \begin{align*}
                \hat{\psi}(\cdot, \eta(\cdot)) &= \zeta,\\
                \hat{\psi}(\cdot,-h) &= 0.
            \end{align*}
        \end{definition}

        \begin{definition}[Dirichlet-to-Neumann operator]
            \label{def:dirichlet_to_neumann_operator}
            Given $\eta$, we define the Dirichlet-to-Neumann operator $G(\eta)$ as the operator mapping Dirichlet data to non-normalized Neumann data; that is, the operator defined by
            \[
                G(\eta)\zeta \coloneqq (-\eta',1) \cdot \restr{\nabla[H(\eta)\zeta]}{y=\eta}
            \]
            for functions $\map{\zeta}{\R}{\R}$.
        \end{definition}
        
        With \Cref{def:harmonic_extension_operator} in mind, define $\map{\zeta}{\R}{\R}$ by
        \begin{align}
            \label{eq:velocity_potential_trace}
            \zeta \coloneqq \hat{\psi}(\cdot,\eta(\cdot)),
        \end{align}
        that is, the trace of $\hat{\psi}$ on the surface. By our assumptions, then, we have
        \[
            \hat{\psi} = H(\eta)\zeta,
        \]
        and we will use this to reformulate \Cref{eq:bernoulli} in a way that only involves $\zeta$ and $\Psi$. Note that
        \begin{align*}
            \zeta' &= \hat{\psi}_x + \eta' \hat{\psi}_y,\\
            G(\eta)\zeta &= -\eta'\hat{\psi}_x + \hat{\psi}_y,
        \end{align*}
        where the right-hand side is evaluated at $y = \eta(x)$. Inverting these relations and inserting them in \Cref{eq:bernoulli} yields
        \begin{equation}
            \begin{multlined}
                \label{eq:zakharov_main}
                c\left[\frac{\eta'\zeta'+G(\eta)\zeta}{\jb{\eta'}^2} + \Psi_y\right]\\+\frac{(\zeta' + (1,\eta')\cdot\nabla\Psi)^2 +(G(\eta)\zeta+(-\eta',1)\cdot\nabla\Psi)^2}{2\jb{\eta'}^2}\\+ g\eta -\alpha^2\kappa(\eta)=0,
            \end{multlined}
        \end{equation}
        and in a similar fashion, one obtains
        \begin{equation}
            \label{eq:zakharov_kinematic}
            c\eta' + \zeta' + (1,\eta')\cdot\nabla \Psi=0
        \end{equation}
        from the kinematic boundary condition in \Cref{eq:boundary_condition_steady_kinematic}. \Cref{eq:zakharov_kinematic} can be integrated once to yield
        \begin{equation}
            \label{eq:zakharov_kinematic_integrated}
            c\eta + \zeta + \Psi = 0,
        \end{equation}
        where we have assumed decay at infinity. We emphasize that the function $\Psi$ and its derivatives are evaluated at $y = \eta(x)$ in \Cref{eq:zakharov_main,eq:zakharov_kinematic,eq:zakharov_kinematic_integrated}, which we suppress for readability. \Cref{eq:zakharov_main,eq:zakharov_kinematic_integrated} form the Zakharov--Craig--Sulem formulation, which we will combine with a suitable vorticity equation. One may note that the pressure, $p$, has been eliminated from the formulation entirely.
        
        \begin{remark}
            \Cref{eq:zakharov_main} is slightly different than the equation used in \cite{Shatah2013}. The equation in \cite{Shatah2013} can be obtained by inserting the kinematic boundary condition, \Cref{eq:zakharov_kinematic}, into \Cref{eq:zakharov_main}.
        \end{remark}
\section{Functional-analytic setting}
    \label{section:functional_analytic_setting}
    We now focus on proving the existence of a family of small amplitude and small velocity traveling waves with vorticity consisting of a point vortex situated on the $y$-axis. In other words, solutions with vorticity of the form
    \[
        \omega = \varepsilon \delta_\theta,
    \]
    where $0 < \abs{\varepsilon} \ll 1$, $\theta \in (0,1)$ and where we have defined
    \[
        \delta_\theta \coloneqq \delta_{(0,-(1-\theta) h)}.
    \]
    The constant $\theta$ then corresponds to the relative position of the point vortex above the bottom, and the parameter $\varepsilon$, describing the strength of the vortex, will be used as the bifurcation parameter.
    
    We will from here on always assume that $\eta$ is such that $\eta(0) > -(1-\theta)h$, which prevents the surface from touching the point vortex. Furthermore, we will also assume that $\eta(0) < (1-\theta)h$. The reason for this is purely technical (as we will see after \Cref{prop:phi_properties}). For the purpose of accounting for these assumptions, define the set
    \begin{equation}
        \label{eq:definition_of_surface_profile_set}
        \Lambda_\theta \coloneqq \{\eta \in BC(\R) : \eta > -h, \abs{\eta(0)} < (1-\theta)h\},
    \end{equation}
    whose intersection $\Lambda_\theta \cap H^s(\R)$ is open in $H^s(\R)$ for any $s > 1/2$, by the Sobolev embedding $H^s(\R) \hookrightarrow BC(\R)$.
    
    The next proposition describes the stream function that we will use for the rotational part of the velocity; the counterpart of the function $\mathbf{G}$ in \cite{Shatah2013}. While we could use a similar stream function on finite depth, it is more beneficial to work with one that is tailored for finite depth.
 
    \begin{proposition}[Stream function]
        \label{prop:phi_properties}
        Let $\eta \in \Lambda_\theta$ and define $\map{\Phi}{\Omega(\eta)}{\R}$ by
        \[
            \Phi(x,y) \coloneqq \frac{1}{4\pi} \log{\left(\frac{\cosh(\pi x/h) +\cos(\pi(y/h-\theta))}{\cosh(\pi x/h) + \cos(\pi (y/h + \theta))}\right)}.
        \]
        Then $\Phi$ defines a regular distribution, and
        \begin{equation}
            \label{eq:point_vortex_laplace_equation}
            \begin{aligned}
                \Delta \Phi &= \delta_\theta,\\
                \restr{\Phi}{y=0} &= 0,\\
                \restr{\Phi}{y=-h} &= 0.
            \end{aligned}
        \end{equation}
        Moreover, the function $(x,y) \mapsto \Phi(x,y) - \Gamma(x,y+(1-\theta)h)$ is harmonic and satisfies
        \begin{equation}
            \label{eq:phi_gradient}
            \nabla^\perp\left(\Phi - \Gamma(x,y+(1-\theta)h)\right)(0,-(1-\theta)h) = \left(\frac{1}{4h}\cot(\pi \theta), 0\right),
        \end{equation}
        where $\Gamma$ is the Newtonian potential introduced in \Cref{prop:delta_vorticity}.
    \end{proposition}
    \begin{proof}
        We will apply \Cref{thm:greens_functions} (see \Cref{appendix}) to prove this result, and thus need a bijective conformal map from the strip $\R \times (-h,0) \subseteq \C$ to the unit disk $\mathbb{D}$, mapping the point $-i(1-\theta)h$ to the origin. This is done in three steps:
        \[
            \includegraphics{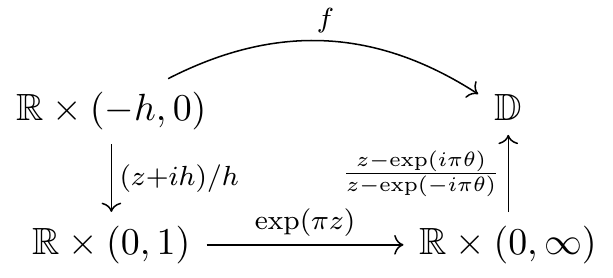}
        \]
        The conformal map for each individual step is well known from elementary complex analysis, see for instance \cite[II.7 and p. 60]{Gamelin2001}. Hence
        \[
            f(z) \coloneqq \frac{e^{\pi (z+ih)/h} - e^{i \pi \theta}}{e^{\pi (z+ih)/h} - e^{-i\pi \theta}}
        \]
        defines the desired map from the strip to the unit disk. By the aforementioned theorem, then,
        \begin{align*}
            \Phi(x,y) &\coloneqq \frac{1}{2\pi} \log(|f(x+iy)|)\\
            &=\frac{1}{4\pi} \log{\left(\frac{\cosh(\pi x/h) +\cos(\pi(y/h-\theta))}{\cosh(\pi x/h) + \cos(\pi (y/h+ \theta))}\right)}
        \end{align*}
        solves \Cref{eq:point_vortex_laplace_equation} in $\R \times (-h,0)$. Because $f$ extends to a meromorphic function on $\C$, it is immediate that $\Phi$ also solves \Cref{eq:point_vortex_laplace_equation} in $\Omega(\eta)$ (recall that $\eta \in \Lambda_\theta$).
        
        Finally, we have
        \begin{align*}
            \nabla^\perp\left(\Phi - \Gamma(x,y+(1-\theta)h)\right)(0,\theta) &= \frac{i}{4\pi}\overline{\left(\frac{f''(i\theta)}{f'(i\theta)}\right)}\\
            &=\left(\frac{1}{4h}\cot(\pi\theta),0\right)
        \end{align*}
        by the final part of the same theorem, which will be important for the asymptotic velocity of the traveling waves that we shall obtain in \Cref{thm:existence_point_vortex}.
    \end{proof}
    Note that the stream function $\Phi$ introduces a ``mirror vortex'' at $(0, (1-\theta)h)$, and moreover is $2h$-periodic in the $y$-direction. This is the reason for the limitation on the height of the surface profiles in the set $\Lambda_\theta$ defined in \Cref{eq:definition_of_surface_profile_set}.
    
    The next proposition is crucial, because the traces of $\Phi$ and its derivatives on the surface enter in the Zakharov--Sulem--Craig formulation of the problem. Having an explicit expression for $\Phi$ enables us to prove the proposition in a quite direct way.
    
    \begin{proposition}
        \label{prop:phi_trace}
        Suppose that $\eta \in H^s(\R) \cap \Lambda_\theta$, where $s > \frac{1}{2}$. Then
        \begin{gather*}
            \Phi(\cdot,\eta(\cdot)) \in H^s(\R),\\
            \nabla^\perp\Phi(\cdot,\eta(\cdot)) \in H^s(\R)^2.
        \end{gather*}
        Moreover, the dependence on $\eta$ is analytic.
    \end{proposition}
    \begin{proof}
        We will only treat $\Phi$, as the argument for the derivative is similar. Observe that it is sufficient to consider the function defined by
        \begin{equation}
            \label{eq:phi_trace_sufficient}
            x \mapsto \log(1+\cos(\eta(x)-\theta)\sech(x))
        \end{equation}
        for $\eta \in H^s(\R)$ such that $\abs{\eta(0)} < \pi-\theta$, where $\theta \in (0,\pi)$. Since
        \[
            \cos(\eta(x)-\theta)\sech(x) = (\cos(\eta(x)-\theta)-\cos(\theta))\sech(x)+\cos(\theta)\sech(x)
        \]
        it follows by \cite[Theorem 4 of 5.5.3]{Runst1996}, $\sech \in H^s(\R)$ and $H^s(\R)$ being an algebra that the function in \Cref{eq:phi_trace_sufficient} lies in $H^s(\R)$ and that the dependence on $\eta$ is analytic. Another application of the result in \cite{Runst1996} then yields the desired result.
    \end{proof}

    As we have seen, because of the reliance on the stream function and the operators $H(\eta)$ and $G(\eta)$, a central problem is the solution of the Laplace equation,
    \begin{equation}
        \begin{gathered}
            \label{eq:laplace}
            \Delta \hat{\psi} = 0 \quad \text{in $\Omega(\eta)$},\\
            \restr{\hat{\psi}}{y=\eta} = \zeta,\quad \restr{\hat{\psi}}{y=-h}= 0
        \end{gathered}
    \end{equation}
    on the fluid domain $\Omega(\eta)$, given $\eta$ and $\zeta$. We have the following theorem, which is adapted from Corollary 2.44 in \cite{Lannes2013}, and which establishes both existence and uniqueness to \Cref{eq:laplace} in suitable Sobolev spaces. Functions on the surface will be identified with functions on the real line as in \Cref{eq:velocity_potential_trace}.
    
    \begin{theorem}[Well-posedness of the Laplace equation \cite{Lannes2013}]
        \label{eq:laplace_well_posed}
        Suppose that $\eta \in H^s(\R) \cap \Lambda_\theta$ for some $s > 3/2$, and that $\zeta \in H^{3/2}(\R)$. Then \Cref{eq:laplace} has a unique solution in $H^2(\Omega(\eta))$.
    \end{theorem}
    \begin{remark}
        While the natural setting for the velocity potential or the stream function on infinite depth is that of homogeneous Sobolev spaces, used in both \cite{Shatah2013} and \cite{Lannes2013}, this is not the case for the stream function on finite depth. Because we require $\hat{\psi}$ to be equal to a constant at the bottom, it must necessarily be the case that $\hat{\psi}$ tends to the same constant at infinity. Otherwise, because of the finite depth, $\hat{\psi_y}$ would not decay at infinity (in the sense that $\lim_{\abs{(x,y)}\to \infty} \psi_y(x,y) = 0$), and therefore not describe a localized wave\footnote{One could say that such a wave is localized if the limit exists and is different from zero, but this does not yield any new waves (only a change in the frame of reference).}.
    \end{remark}
    
    \Cref{eq:laplace_well_posed} enables us to define the harmonic extension operator described in \Cref{def:harmonic_extension_operator} as an operator $H^{3/2}(\R) \to H^2(\Omega(\eta))$, and using this, defining the Dirichlet-to-Neumann operator. We refer the reader to \cite{Lannes2013}, which is a rich source of results for these operators also in a more general setting. The results there are proved for the velocity potential, but should be adaptable for the stream function.\footnote{Compare with Theorems 3.49 and A.13 in \cite{Lannes2013} for the case of infinite depth, where the boundary conditions for the Laplace equation for the stream function and velocity potential coincide.} The below theorem is an amalgamation of parts from Corollary 2.40 and Theorems 3.15 and A.11 in \cite{Lannes2013}. See also \cite{Shatah2008}.
    
    \begin{theorem}[Boundary operators \cite{Lannes2013}]
        \label{thm:boundary_operators}
        Let $s > 3/2$ and suppose that $\eta \in H^s(\R) \cap \Lambda_\theta$. Then the harmonic extension operator $H(\eta)$ and the Dirichlet-to-Neumann operator $G(\eta)$ are members of $B(H^{3/2}(\R),H^2(\Omega(\eta))$ and $B(H^s(\R),H^{s-1}(\R))$, respectively. The norms of these operators are uniformly bounded on subsets of $H^s(\R) \cap \Lambda_\theta$ that are bounded in the norm on $H^s(\R)$. Moreover, the map $G(\cdot)\zeta$ is analytic for fixed $\zeta \in H^s(\R)$.
    \end{theorem}
    
    In the same setting as in \Cref{thm:boundary_operators}, the curvature of the surface is well defined:
    
    \begin{proposition}[Curvature]
        \label{prop:curvature}
        The curvature operator $\kappa$ is well defined as an operator $H^s(\R) \to H^{s-2}(\R)$ for any $s > 3/2$. Moreover, the map is analytic.
    \end{proposition}
    \begin{proof}
        Observe that the function $\map{f}{\R}{\R}$ defined by $f(x)=x\jb{x}^{-1}$ is smooth and satisfies $f(0)=0$. As $s-1 > \frac{1}{2}$, the result \cite[Theorem 4 of 5.5.3]{Runst1996} ensures that $f(\eta') \in H^{s-1}(\R)$. Since $f$ is also analytic, $\kappa$ is analytic by the same result.
    \end{proof}
        
    There is one thing we have not yet looked at, namely the vorticity equation \Cref{eq:vorticity_transport_equation_weak_steady}. Recalling \Cref{eq:velocity_in_terms_of_stream_function}, we will consider velocity fields of the form
    \begin{equation}
        \label{eq:velocity_fields_considered_point_vortex}
        w = \nabla^\perp\left(H(\eta)\zeta + \varepsilon \Phi\right).
    \end{equation}
    From \Cref{prop:delta_vorticity} we know that the part of the stream function that is generated by the point vortex at $(0,-(1-\theta)h)$ is given by the Newtonian potential
    \[
        \varepsilon\Gamma(x,y+(1-\theta)h), 
    \]
    whence \Cref{eq:velocity_fields_considered_point_vortex} reduces to
    \[
        (c,0) = \nabla^\perp[H(\eta)\zeta](0,-(1-\theta)h) + \varepsilon\left(\frac{1}{4h}\cot(\pi\theta),0\right),
    \]
    by \Cref{eq:phi_gradient} in \Cref{prop:phi_properties}.
    
    In particular, this means that any solution necessarily must satisfy
    \[
        [H(\eta)\zeta]_x(0,-(1-\theta)h) = 0.
    \]
    For simplicity, we choose to look for $\eta, \zeta$ in appropriately chosen subspaces of $H^s(\R)$, such that this condition is automatically satisfied. Specifically, define
    \[
            H_\text{even}^s(\R) \coloneqq \{f \in H^s(\R) : \text{$f$ is even}\},
    \]
    which is closed in $H^s(\R)$, and therefore a Hilbert space in the inherited norm. We mention that it is still an open question whether asymmetric traveling waves exist. However, it is known that, in many situations, certain properties imply symmetry (see for instance \cite{Constantin2007,Hur2008}). Furthermore, under suitable assumptions, all symmetric waves are traveling waves \cite{Ehrnstroem2009}.
    
    Assume now that $\eta \in H_\text{even}^s(\R) \cap \Lambda_\theta$, with $s > 3/2$, and that $\zeta \in H_\text{even}^{3/2}(\R)$. Then it must necessarily be the case that $H(\eta)\zeta$ is even in $x$. Hence $[H(\eta)\zeta]_x$ vanishes along the $y$-axis and so the vorticity equation reduces further to
    \begin{equation}
        \label{eq:vorticity_equation_point_vortex}
        c =  c_1\varepsilon - [H(\eta)\zeta]_y(0,-(1-\theta)h), \quad \text{where $c_1 \coloneqq \frac{1}{4h}\cot(\pi\theta)$.}
    \end{equation}
    Observe also that the Dirichlet-to-Neumann operator $G(\eta)$ is well defined as an operator $H_\text{even}^s(\R) \to H_\text{even}^{s-1}(\R)$ for $\eta \in H^s_\text{even} \cap{\Lambda_\theta}$ and $s > 3/2$, and that $\kappa$ can be viewed as an operator $H_\text{even}^s(\R) \to H_\text{even}^{s-2}(\R)$.
    
    \begin{remark}
        One has to be careful with claims about the solution set when $\varepsilon = 0$. \Cref{eq:vorticity_equation_point_vortex} of course actually only needs to be satisfied if $\varepsilon \neq 0$. This means that if we impose \Cref{eq:vorticity_equation_point_vortex}, then we lose the trivial set of solutions
        \[
            (\eta, \zeta, c, \varepsilon) \in \{0\} \times \{0\} \times \R \times \{0\}
        \]
        for the other equations, except for the point $(0,0,0,0)$. This should be kept in mind in any claims of uniqueness.
    \end{remark}
    
    For convenience, define now the spaces
    \begin{align*}
        X^s &\coloneqq H_\text{even}^s(\R) \times H_\text{even}^s(\R) \times \R,\\
        Y^s &\coloneqq H_\text{even}^{s-2}(\R) \times H_\text{even}^s(\R) \times \R
        \intertext{and the set}
        U_\theta^s &\coloneqq \{(\eta,\zeta,c) \in X^s: \eta \in \Lambda_\theta\} \subseteq X^s,
    \end{align*}
    which accounts for the limitations on $\eta$.
    
    We proceed to introduce three maps that together will form the basis for our argument. For $s > 3/2$ we define $\map{F_1}{U_\theta^s \times \R}{H_\text{even}^{s-2}(\R)}$ by
    \begin{multline*}
        F_1(\eta, \zeta, c, \varepsilon) = c\left[\frac{\eta'\zeta'+G(\eta)\zeta}{\jb{\eta'}^2}+\varepsilon\Phi_y\right]\\ +\frac{(\zeta' + \varepsilon(1,\eta') \cdot \nabla \Phi)^2 +(G(\eta)\zeta +\varepsilon(-\eta',1)\cdot \nabla \Phi)^2}{2\jb{\eta'}^2}
        + g\eta -\alpha^2\kappa(\eta),
    \end{multline*}
    the map $\map{F_2}{U_\theta^s \times \R}{H_\text{even}^s(\R)}$ by
    \[
        F_2(\eta,\zeta,c,\varepsilon) = c\eta + \zeta+\varepsilon\Phi,
    \]
    and finally the map $\map{F_3}{U_\theta^s \times \R}{\R}$ by
    \begin{equation}
        \label{eq:F3_definition}
        F_3(\eta,\zeta,c,\varepsilon) = c - c_1 \varepsilon + [H(\eta)\zeta]_y(0,-(1-\theta)h).
    \end{equation}
    In all of these definitions, we really mean the traces of $\Phi$ and its derivatives on the surface. The pointwise evaluation in the second term of \Cref{eq:F3_definition} is allowed because $H(\eta)\zeta$ is harmonic. It should be clear that all three maps $F_1, F_2, F_3$ are smooth.
     
    We can now define $\map{F}{U_\theta^s \times \R}{Y^s}$ by
    \[
        F \coloneqq (F_1, F_2, F_3),
    \]
    and our task will then be to find solutions of the equation
    \begin{equation}
        \label{eq:zcs_formulation_point_vortex}
        F(\eta,\zeta,c,\varepsilon)=0.
    \end{equation}
    One may immediately note that $F(0,0,0,0)=0$, so that the origin is a trivial solution. It will turn out that in a small neighborhood of the origin in $X^s \times \R$, there is a unique curve of nontrivial solutions parametrized by the vortex strength parameter $\varepsilon$. 
    
\section{Local bifurcation}
    \label{section:local_bifurcation}
    We can now finally state and prove the following theorem, establishing the existence of small, localized, traveling wave solutions with a point vortex. For this, we will use an implicit function theorem argument on $F$. Note that while we do not apply Crandall--Rabinowitz theorem \cite{Crandall1971}, the situation is very much in the spirit of that theorem. We bifurcate from the family of trivial waves described in the remark after \Cref{eq:vorticity_equation_point_vortex} by introducing the vorticity equation.

    \begin{theorem}[Traveling waves with a point vortex]
        \label{thm:existence_point_vortex}
        Let $s > 3/2$ and $\theta \in (0,1)$. Then there exists an open interval $I \ni 0$ and a $C^\infty$-curve
        \[
            \arraycolsep=0.03\textwidth
            \begin{array}{ccc}
                I & \to & (H_\text{even}^s(\R) \cap \Lambda_\theta) \times H_\text{even}^s(\R) \times \R \times \R\\
                \varepsilon & \mapsto & (\eta(\varepsilon),\zeta(\varepsilon),c(\varepsilon), \varepsilon)
            \end{array}
        \]
        of solutions to the Zakharov--Craig--Sulem formulation, \Cref{eq:zcs_formulation_point_vortex}, for a point vortex of strength $\varepsilon$ situated at $(0,-(1-\theta)h)$. The solutions fulfil
        \begin{equation}
            \label{eq:asymptotic_point_vortex}
            \begin{aligned}
                \eta(\varepsilon)&= \eta_2\varepsilon^2 + O(\varepsilon^4),\\
                \zeta(\varepsilon)&= \zeta_3 \varepsilon^3 + O(\varepsilon^4),\\
                c(\varepsilon)&= c_1\varepsilon + c_3 \varepsilon^3 + O(\varepsilon^4),
            \end{aligned}
        \end{equation}
        in their respective spaces as $\varepsilon \to 0$, where $\eta_2 \in H_\text{even}^s(\R)$ is defined by
        \begin{align*}
            \eta_2 \coloneqq -(g-\alpha^2\partial_x^2)^{-1}\chi, \quad \chi \coloneqq c_1 \Phi_y(\cdot,0) + \frac{1}{2}\Phi_y(\cdot,0)^2,
        \end{align*}
        and where
        \begin{align*}
            \zeta_3 &\coloneqq -\eta_2(c_1 + \Phi_y(\cdot,0)),\\
            c_3 &\coloneqq -[H(0)\zeta_3]_y(0,-(1-\theta)h),
        \end{align*}
        with $c_1$ as in \Cref{eq:vorticity_equation_point_vortex}, $\Phi$ as in \Cref{prop:phi_properties} and $H$ as in \Cref{def:harmonic_extension_operator}.
        
        Moreover, there is a neighborhood of the origin in $U_\theta^s \times \R$ such that this curve describes all solutions to $F(\eta, \zeta, c,\varepsilon)=0$ in that neighborhood.
    \end{theorem}
    \begin{proof}
        As remarked at the end of \Cref{section:functional_analytic_setting}, the origin is a trivial solution. In order to apply the implicit function theorem, we require the first partial derivatives of $F$ at this point. A direct calculation yields
        \begin{equation}
            \label{eq:point_vortex_map_derivative_X}
            D_XF(0,0,0,0) = \begin{bmatrix}
                g - \alpha^2 \partial_x^2 & 0 & 0\\
                0 & I_{H_\text{even}^s(\R)} & 0\\
                0 & [H(0)\cdot]_y(0,-(1-\theta)h) & 1
            \end{bmatrix},
        \end{equation}
        where the subscript $X$ denotes the partial derivative with respect to the variable $(\eta, \zeta,c)$ in $X^s$.
        
        Now, every operator on the diagonal of $D_XF(0,0,0,0)$ is an isomorphism. Indeed, the operator
        \[
            \map{[g-\alpha^2\partial_x^2]}{H_\text{even}^s(\R)}{H_\text{even}^{s-2}(\R)}
        \]
        corresponds to the Fourier multiplier $g +\alpha^2 \xi^2$. Since $g,\alpha^2 >0$, this operator is invertible, with inverse corresponding to the multiplier $(g + \alpha^2 \xi^2)^{-1}$. The other two operators on the diagonal are identity operators, and therefore trivially invertible. Hence $D_XF(0,0,0,0) \in B(X^s, Y^s)$ is also an isomorphism.
        
        Thus we can use the implicit function theorem to conclude that there is an open interval $I$ containing zero, an open set $V \subseteq U_\theta^s$ containing $(0,0,0)$, and a map $f \in C^\infty(I,V)$ such that for $(\eta,\zeta,c,\varepsilon) \in V \times I$, we have
        \[
            F(\eta,\zeta,c,\varepsilon) = 0 \iff (\eta,\zeta,c) = f(\varepsilon).
        \]
        Furthermore, we obtain
        \[
            Df(0) = -D_XF(0,0,c_1,0)^{-1}D_\varepsilon F(0,0,c_1,0) = \begin{bmatrix}
                0\\
                0\\
                c_1
            \end{bmatrix},
        \]
        which yields the first order terms in \Cref{eq:asymptotic_point_vortex}. The higher-order terms can be obtained by inserting expansions for $\eta(\varepsilon), \zeta(\varepsilon)$ and $c(\varepsilon)$ into the equation $F(\eta(\varepsilon), \zeta(\varepsilon), c(\varepsilon), \epsilon)=0$. This concludes the proof of the theorem.
    \end{proof}
    \begin{remark}
        Because \Cref{thm:existence_point_vortex} holds for any $s > 3/2$, we can get arbitrarily high regularity on the solutions, by possibly making the interval $I$ smaller. We have not been able to conclude that they are smooth, however, since the interval could possibly shrink to a point as $s \to \infty$.
    \end{remark}
    
    Observe that, because $c_1$ changes sign at $\theta = 1/2$, the direction in which the waves obtained in \Cref{thm:existence_point_vortex} will travel (for small $\varepsilon$) depends on where the point vortex is in relation to the line $y = -h/2$. This does not come into play for waves on infinite depth. Note that if $\theta = 1-1/h$ (when $h > 1$) then
    \[
        c_1 = -\frac{1}{4\pi} + O(1/h^2)
    \]
    as $h \to \infty$, which is in agreement with what was found in \cite{Shatah2013} for a point vortex situated at $(0,-1)$ on infinite depth.
    
    Since $c_1$ vanishes when $\theta=1/2$, also the next term in the expansion for $c(\varepsilon)$ is of interest. We gave an expression for $c_3$ in \Cref{thm:existence_point_vortex}, but have not determined its sign yet. We will treat the sign of $c_3$ after \Cref{theorem:asymptotic_behavior_of_eta}, which establishes some properties of the function $\eta_2$.

    Written out, we have
    \begin{equation}
        \label{eq:chi_written_out}
        \chi(x)=\frac{1}{8h^2}\frac{1+\cos(\pi\theta)\cosh(\pi x/h)}{(\cosh(\pi x/h) + \cos(\pi \theta))^2},
    \end{equation}
    for the function $\chi$ defined in \Cref{thm:existence_point_vortex}. We will have use for the fact that $\chi$ has an elementary antiderivative $\ad{\chi}{1}$ and a double antiderivative $\ad{\chi}{2}$ given by
    \begin{equation}
        \label{eq:chi_antiderivatives}
        \begin{aligned}
            \ad{\chi}{1}(x)&=\frac{1}{8\pi h}\frac{\sinh(\pi x/h)}{\cosh(\pi x/h)+\cos(\pi\theta)},\\
            \ad{\chi}{2}(x)&=\frac{1}{8\pi^2}\log(\cosh(\pi x/h)+\cos(\pi\theta)),
        \end{aligned}
    \end{equation}
    respectively. While there in general seems to be no nice closed form of the leading order surface profile
    \[
        \eta_2 = -(g-\alpha^2 \partial_x^2)^{-1}\chi
    \]
    obtained in \Cref{thm:existence_point_vortex}, we can still give some of its properties. An immediate one is that $\eta_2$ is smooth. In \Cref{prop:eta_point_vortex_explicit} we give a series expansion for $\eta_2$ in powers of $e^{-\pi\abs{x}/h}$. Furthermore, perhaps more surprisingly, we \emph{can} find an explicit expression for $\eta_2$ in terms of elementary functions whenever
    \begin{equation}
        \label{eq:definition_of_m}
        m \coloneqq \frac{\sqrt{g}h}{\pi\alpha}.
    \end{equation}
    is a natural number. If $m \in \N$, then $e^{\pm\sqrt{g}x/\alpha}=e^{\pm m\pi x/h}$ are integral powers of $e^{\pm \pi x/h}$, which would appear on the right side of \Cref{eq:chi_written_out} if we had written out $\cosh(\pi x/h)$ and $\sinh(\pi x/h)$. Since $x \mapsto e^{\pm\sqrt{g}x/\alpha}$ spans the kernel of $g-\alpha^2\partial_x^2$, this explains integral values of $m$ being special.
    
    Before we state \Cref{prop:eta_point_vortex_explicit} and \Cref{theorem:asymptotic_behavior_of_eta}, we need a lemma to simplify some expressions.
    \begin{lemma}
        \label{lemma:explicit_expression_for_eta_coefficients}
        For $m \in (0,\infty)\setminus \N$ and $\theta \in (0,1)$, we have
        \begin{equation}
            \label{eq:eta_coefficient_explicit_not_integer}
            \frac{1}{m} + 2m\sum_{k=1}^\infty (-1)^k \frac{\cos(k\pi\theta)}{m^2-k^2}=\pi\frac{\cos(m\pi\theta)}{\sin(m\pi)},
        \end{equation}
        which is equal to
        \begin{equation}
            \label{eq:eta_coefficient_explicit_0_1_integral}
            \int_0^\infty y^{m-1}\frac{\cos(\pi\theta)y+1}{y^2+2\cos(\pi\theta)y+1}\,dy
        \end{equation}
        whenever $m \in (0,1)$. Furthermore, for $m \in \N$,
        \begin{equation}
            \label{eq:eta_coefficient_explicit_integer}
            \frac{1}{m}+2m\sum_{\substack{k=1\\k\neq m}}^\infty(-1)^k\frac{\cos(k\pi\theta)}{m^2-k^2} = -(-1)^m\left(\frac{\cos(m\pi\theta)}{2m} +\pi\theta \sin(m\pi \theta)\right).
        \end{equation}
    \end{lemma}
    \begin{proof}
        Both sides of \Cref{eq:eta_coefficient_explicit_not_integer} define meromorphic functions in $m$ on $\C$ with simple poles in the points $\Z \times \{0\}$. Moreover, they are both equal to the integral in \Cref{eq:eta_coefficient_explicit_0_1_integral} when $m \in (0,1)$, which can be seen by calculating the integral with both the residue theorem (around a keyhole contour) and a Laurent series expansion of the integrand. Since the interval consists of non-isolated points, we have equality on all of $\C$. Finally, \Cref{eq:eta_coefficient_explicit_integer} follows from \Cref{eq:eta_coefficient_explicit_not_integer} by taking limits.
    \end{proof}
    
    \begin{proposition}[Expansion for $\eta_2$]
        \label{prop:eta_point_vortex_explicit}
        If the number $m$ in \Cref{eq:definition_of_m} satisfies $m \in (0,\infty)\setminus \N$, then the leading order term of the surface profile from \Cref{thm:existence_point_vortex} is given by
        \begin{multline*}
            \eta_2(x)=\frac{1}{8\pi^2\alpha^2}\bigg[\log(1+2\cos(\pi\theta)e^{-\pi\abs{x}/h}+e^{-2\pi\abs{x}/h})\\-\pi\frac{\cos(m\pi\theta)}{\sin(m\pi)}e^{-\sqrt{g}\abs{x}/\alpha} + 2m^2\sum_{k=1}^\infty (-1)^k\frac{\cos(k\pi\theta)}{k(m^2-k^2)}e^{-k\pi\abs{x}/h}\bigg],
        \end{multline*}
        while if $m \in \N$, then
        \begin{multline*}
            \eta_2(x)=\frac{1}{8\pi^2\alpha^2}\biggl[\log(1+2\cos(\pi\theta)e^{-\pi\abs{x}/h}+e^{-2\pi\abs{x}/h})\\+(-1)^m\left(\frac{3\cos(m\pi\theta)}{2m} +\pi\theta \sin(m\pi \theta)+\cos(m\pi\theta) \frac{\pi \abs{x}}{h}\right)e^{-\sqrt{g}\abs{x}/\alpha} \\+ 2m^2\sum_{\substack{k=1\\k\neq m}}^\infty (-1)^k\frac{\cos(k\pi\theta)}{k(m^2-k^2)}e^{-k\pi\abs{x}/h}\biggr].
        \end{multline*}
        These series converge uniformly, and, excluding the origin, so do the series for the first derivative. Moreover, when $m \in \N$, the function $\eta_2$ is given explicitly in terms of elementary functions by
        \begin{multline*}
            \eta_2(x) = \frac{1}{8\pi^2\alpha^2}\biggl[\frac{1}{m} + 2\sum_{k=1}^{m-1} (-1)^{m-k}\frac{\cos((m-k)\pi\theta)}{k}\cosh((m-k)\pi x/h)\\+r(e^{\pi x/h})+r(e^{-\pi x/h})\biggr],
        \end{multline*}
        where $\map{r}{(0,\infty)}{\R}$ is defined by
        \begin{multline*}
            r(x) \coloneqq \frac{1}{2}(-1)^m \cos(m\pi\theta)x^{-m}\log(1+2\cos(\pi\theta)x + x^2)\\+(-1)^m\sin(m\pi\theta)x^{-m}(\arctan(\cot(\pi\theta)+\csc(\pi\theta)x)-\pi(1/2-\theta)).
        \end{multline*}
    \end{proposition}
    \begin{proof}
        It follows from
        \[
            \mathscr{F}(e^{-a\abs{\cdot}})(\xi) = \sqrt{\frac{2}{\pi}}\frac{a}{a^2+\xi^2}, \quad a > 0,
        \]
        and the definition of $\eta_2$, that we may write $\eta_2$ as the convolution
        \begin{align}
            \eta_2(x)&=-\frac{1}{2\alpha\sqrt{g}}(e^{-\sqrt{g}\abs{\cdot}/\alpha}*\chi)(x)\notag\\
            &=-\frac{1}{2\alpha \sqrt{g}}\left(J(x,\chi) + J(-x,\chi)\right), \label{eq:eta_convolution_written_out}
        \end{align}
        where
        \[
            J(x,\chi) \coloneqq e^{-\sqrt{g}x/\alpha}\int_{-\infty}^x e^{\sqrt{g}y/\alpha}\chi(y)\,dy.
        \]
        Equivalently
        \begin{align}
            \label{eq:eta_convolution_first_antiderivative}
            \eta_2(x)& =\frac{1}{2\alpha^2}\left(J(x,\ad{\chi}{1}) + J(-x,\ad{\chi}{1})\right)\\
            \label{eq:eta_convolution_second_antiderivative}
            &=\frac{1}{\alpha^2}\ad{\chi}{2}(x)-\frac{\sqrt{g}}{2\alpha^3}\left(J(x,\ad{\chi}{2})+J(-x,\ad{\chi}{2})\right)
        \end{align}
        through integration by parts, where $\ad{\chi}{1}$ and $\ad{\chi}{2}$ are the antiderivatives defined in \Cref{eq:chi_antiderivatives}.
        
        We first use \Cref{eq:eta_convolution_first_antiderivative} to obtain an explicit expression for $\eta_2$ when $m \in \N$. By using the substitution $x \mapsto e^{\pi x/h}$, we find that
        \[
            J(x,\ad{\chi}{1}) = \frac{1}{8\pi^2}f_1(e^{\pi x/h}), \quad f_1(x)\coloneqq x^{-m}\int_0^x z^{m-1} \frac{z^2-1}{z^2+2\cos(\pi\theta)z+1}\,dz.
        \]
        The fraction in the integrand in the definition of $f_1$ has partial fraction decomposition
        \[
            \frac{z^2-1}{z^2+2\cos(\pi\theta)z+1} = 1- \frac{e^{i\pi\theta}}{z+e^{i\pi\theta}} - \frac{e^{-i\pi\theta}}{z+e^{-i\pi\theta}},
        \]
        and since
        \[
            z^{m-1}\frac{a}{z+a} = -\frac{(-a)^m}{z+a} - \sum_{k=0}^{m-2} (-a)^{m-k-1}z^k,\quad a \in \C, z \neq -a,
        \]
        this means that
        \begin{multline*}
            f_1(x) = \frac{1}{m} + (-1)^m e^{im\pi\theta}x^{-m}\log(x+e^{i\pi\theta})+(-1)^me^{-im\pi\theta}x^{-m}\log(x+e^{-i\pi\theta})\\+2(-1)^m\pi\theta\sin(m\pi\theta)x^{-m}+2\sum_{k=1}^{m-1}\frac{(-1)^k\cos(k\pi\theta)}{m-k}x^{-k},
        \end{multline*}
        where $\log(\cdot)$ denotes the principal branch of the logarithm. The result now follows by using the identity
        \begin{multline*}
            \log(x+e^{i\pi\theta}) = \frac{1}{2}\log(1+2\cos(\pi\theta)x+x^2)\\-i(\arctan(\cot(\pi\theta)+\csc(\pi\theta)x)-\pi/2),
        \end{multline*}
        valid for all $x \in \R$.
        
        For the series representation of $\eta_2$, we use \Cref{eq:eta_convolution_second_antiderivative}, because this leads to a series that converges more rapidly. We will assume that $m \in (0,\infty) \setminus \N$; the case for $m \in \N$ is similar, except that one needs to use \Cref{eq:eta_coefficient_explicit_integer} instead of \Cref{eq:eta_coefficient_explicit_not_integer}. We use the same substitution as before to arrive at
        \[
            J(x,\ad{\chi}{2}) = \frac{\alpha}{8\pi^2 \sqrt{g}} f_2(e^{\pi x/h}),
        \]
        where $\map{f_2}{(0,\infty)}{\R}$ is defined by
        \[
            f_2(x)\coloneqq mx^{-m}\int_0^x z^{m-1}\log((z^{-1}+z)/2+\cos(\pi\theta))\,dz.
        \]
        One may check that one has
        \[
            \log((z^{-1}+z)/2+\cos(\pi\theta)) =-\log(2)-\log(z) -2 \sum_{k=1}^\infty (-1)^k\frac{\cos(k\pi\theta)}{k}z^k
        \]
        for $z \in (0,1)$ and
        \[
            \log((z^{-1}+z)/2+\cos(\pi\theta)) =-\log(2)+\log(z) -2 \sum_{k=1}^\infty (-1)^k\frac{\cos(k\pi\theta)}{k}z^{-k}
        \]
        for $z \in (1,\infty)$.
        
        It then follows by termwise integration that
        \[
            f_2(x) = \frac{1}{m}-\log(2)-\log(x)-2m\sum_{k=1}^\infty(-1)^k \frac{\cos(k\pi\theta)}{k(m+k)}x^k
        \]
        on $(0,1]$ (the endpoint is Abel's theorem \cite[Theorem 17.14]{Markushevich1965}), and that
        \begin{align*}
            f_2(x) &= f_2(1)x^{-m} + mx^{-m}\int_1^x z^{m-1}\log((z+z^{-1})/2+\cos(\pi\theta))\,dz\\
            &=\begin{multlined}[t]-\frac{1}{m}-\log(2)+\log(x)-2m\sum_{k=1}^\infty (-1)^k \frac{\cos(k\pi\theta)}{k(m-k)}x^{-k}\\-\left(\frac{2}{m}+4m\sum_{k=1}^\infty (-1)^k\frac{\cos(k\pi\theta)}{m^2-k^2}\right)x^{-m}.
            \end{multlined}
        \end{align*}
        for $x \in [1,\infty)$. Employing \Cref{eq:eta_convolution_second_antiderivative}, we find that $\eta_2$ is given by
        \[
            \begin{multlined}
                \eta_2(x)=\frac{1}{8\pi^2\alpha^2}\bigg[\log(\cosh(\pi x/h)+\cos(\pi\theta))-\pi\abs{x}/h+\log(2)\\-\left(\frac{1}{m}+2m\sum_{k=1}^\infty (-1)^k\frac{\cos(k\pi\theta)}{m^2-k^2}\right)e^{-\sqrt{g}\abs{x}/\alpha}\\+ 2m^2\sum_{k=1}^\infty (-1)^k\frac{\cos(k\pi\theta)}{k(m^2-k^2)}e^{-k\pi\abs{x}/h}\bigg]
            \end{multlined}
        \]
        for all $x \in \R$, by using that $\eta_2$ is even and observing that for $x \geq 0$ we have $e^{\pi x/h} \in [1,\infty)$ and $e^{-\pi x/h} \in (0,1]$. If we now apply \Cref{eq:eta_coefficient_explicit_not_integer} from \Cref{lemma:explicit_expression_for_eta_coefficients} in order to get a closed-form expression for the coefficient in front of $e^{-\sqrt{g}\abs{x}/\alpha}$, we arrive at the desired expansion.
    \end{proof}
    \begin{remark}
        The only obstacle to convergence of the series given in \Cref{prop:eta_point_vortex_explicit} is the origin; thanks to the exponential factor $e^{-k\pi \abs{x}/h}$, the convergence is rapid away from the origin. It should also be noted that, while \Cref{eq:chi_written_out} seems to suggest that $\eta_2$ should be expandable in a series of powers of $\sech(\pi x/h)$ by equating coefficients in the differential equation defining it, this seems to lead to a series that does not converge. We have kept the series expansion for $\eta_2$ also when $m \in \N$, because the expression in terms of elementary functions is unwieldy, and prone to numerical errors even for small values of $m$.
    \end{remark}
    
    The expressions found in \Cref{prop:eta_point_vortex_explicit} have well defined pointwise limits as $\theta \uparrow 1$ (for $x \neq 0$) and $\theta \downarrow 0$. In particular, when $m =1$ these are given by
    \begin{align*}
        \lim_{\theta \downarrow 0} \eta_2(x) &= \frac{1}{8\pi^2\alpha^2}[1-e^{\pi x/h}\log(1+e^{-\pi x/h}) - e^{-\pi x/h}\log(1+e^{\pi x/h})]\\
        \lim_{\theta \uparrow 1} \eta_2(x) &=\frac{1}{8\pi^2\alpha^2}[1 + e^{\pi x/h}\log{\abs{1-e^{-\pi x/h}}} + e^{-\pi x/h}\log{\abs{1-e^{\pi x/h}}}],
    \end{align*}
    which can can be seen as graphs drawn with thicker lines in \Cref{fig:surface_profile_point_vortex}, together with $\eta_2$ for various values of the parameter $\theta$.
    
    \begin{figure}[htb]
        \centering
        \includegraphics{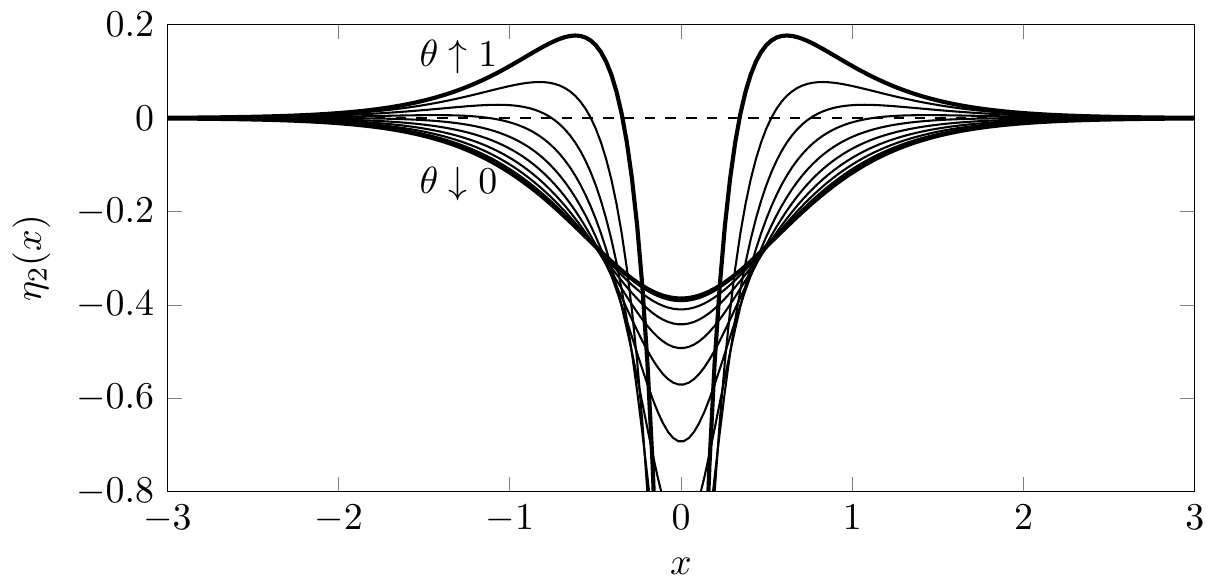}
        \caption{The leading order term $\eta_2$ in $\eta(\varepsilon)$, with $h=1, \alpha^2=1/(8\pi^2), m=1$. The values of $\theta$ shown are $\theta = 0.1,0.2,\ldots,0.9$, together with the thicker lower and upper limits $\theta \downarrow 0$ and $\theta \uparrow 1$.}
        \label{fig:surface_profile_point_vortex}
    \end{figure}
    
    We see from \Cref{fig:surface_profile_point_vortex} that one gets a depression at the origin, which becomes more pronounced the closer the point vortex is situated to the surface. The profile when the point vortex is close to the surface is very similar to the profile for the infinite depth case, found in \cite{Shatah2013}. However, a feature which is not seen on infinite depth is that there is a significant difference between the case $\theta \leq 1/2$ and the case $\theta > 1/2$ (in addition to the changing sign of $c_1$). For $\theta \leq 1/2$ there is a single trough at the origin, and $\eta_2$ is everywhere strictly negative. When $\theta > 1/2$ one in addition gets crests on either side of the origin. As we can see from \Cref{fig:surface_profile_point_vortex}, the positions of these crests depend on the position of the point vortex.
    
    Some of what we have just discussed is not limited to the specific choice of constants that are used in \Cref{fig:surface_profile_point_vortex}, and for which \Cref{prop:eta_point_vortex_explicit} yields an explicit expression for $\eta_2$. We will see that $m=1$ plays a special role in the asymptotic behavior of $\eta_2$, however.  More precisely, we have the following theorem:
    \begin{theorem}[Properties of $\eta_2$]
        \label{theorem:asymptotic_behavior_of_eta}
        The leading order surface term $\eta_2$ always satisfies $\eta_2(0) < 0$ and $\eta_2''(0) > 0$, meaning that the origin is a depression. When $\theta \leq 1/2$, the function $\eta_2$ is everywhere negative, and strictly increasing on $[0,\infty)$. For $\theta>1/2$, we have two cases, depending on the number $m$ defined in \Cref{eq:definition_of_m}:
        \begin{enumerate}[(i)]
            \item If $m > 1/(2\theta)$, then $\eta_2(x)$ is positive for sufficiently large $\abs{x}$. In particular, $\eta_2$ has crests on either side of the origin.
            \item If $m \leq 1/(2\theta)$, then $\eta_2(x)$ is negative for sufficiently large $\abs{x}$.
        \end{enumerate}
        Furthermore, $\eta_2$ has the following asymptotic properties for any $\theta \in (0,1)$:
        \begin{enumerate}[(i)]
            \item For $m>1$
                \begin{equation}
                    \label{eq:eta_limit_m_larger_than_one}
                    \lim_{x\to \infty} \eta_2(x) e^{\pi x/h} = -\frac{2}{m^2-1}\frac{\cos(\pi\theta)}{8\pi^2\alpha^2}.
                \end{equation}
            \item If $m=1$, then
                \[
                    \lim_{x\to \infty} \eta_2(x) \frac{e^{\pi x/h}}{\pi x/h} = -\frac{\cos(\pi\theta)}{8\pi^2\alpha^2}.
                \]
            \item For $m < 1$
                \begin{equation}
                    \label{eq:eta_limit_m_less_than_one}
                    \lim_{x\to \infty} \eta_2(x) e^{\sqrt{g}x/\alpha} = -\frac{\pi}{\sin(m\pi)}\frac{\cos(m\pi\theta)}{8\pi^2\alpha^2}.
                \end{equation}
        \end{enumerate}
    \end{theorem}
    \begin{proof}
        We first prove that $\eta_2(0) < 0$ and $\eta_2''(0)>0$, which holds for all values of $m$ and $\theta$. By inserting $x=0$ in \Cref{eq:eta_convolution_written_out}, and using the evenness of $\chi$, we find
        \begin{align*}
            \eta_2(0) &= -\frac{1}{\alpha\sqrt{g}} \int_0^\infty e^{-\sqrt{g}y/\alpha} \chi(y)\,dy\\
            &=-\frac{1}{\alpha^2}\int_0^\infty \underbrace{e^{-\sqrt{g}y/\alpha}\ad{\chi}{1}(y)}_{\text{$>0$ on $(0,\infty)$}}\,dy < 0,
        \end{align*}
        where the second equality follows from integration by parts, and the function $\ad{\chi}{1}$ was defined in \Cref{eq:chi_antiderivatives}. Since $\eta_2=-(g-\alpha^2\partial_x^2)^{-1}\chi$, we also have
        \begin{align*}
            \eta_2''(0) &= \frac{1}{\alpha^2}(g\eta_2(0)+\chi(0))\\
            &=\frac{\sqrt{g}}{2\alpha^3}\int_{-\infty}^\infty e^{-\sqrt{g}\abs{y}/\alpha} (\chi(0) - \chi(y))\,dy\\
            &>0,
        \end{align*}
        as $\chi$ achieves a global maximum at the origin.
        
        Suppose now that $\theta \leq 1/2$. Like in \Cref{prop:eta_point_vortex_explicit}, we use the fact that $\eta_2$ may be written as the convolution
        \begin{equation}
            \label{eq:eta_as_convolution}
            \begin{aligned}
                \eta_2 &= -\frac{1}{2\alpha \sqrt{g}} (e^{-\sqrt{g}\abs{\cdot}/\alpha}*\chi),\\
            \end{aligned}
        \end{equation}
        which shows that $\eta_2$ is strictly negative, since $\chi$ is strictly positive when $\theta \leq 1/2$. Moreover, some manipulations of the above formula shows that we may write the derivative of $\eta_2$ as
        \begin{multline*}
            \eta_2'(x) = -\frac{1}{\alpha\sqrt{g}}\left[\sinh\left(\frac{\sqrt{g}}{\alpha}x\right)\int_x^\infty e^{-\sqrt{g}y/\alpha} \chi'(y)\,dy\right. \\ \left.+ e^{-\sqrt{g}x/\alpha}\int_0^x \sinh\left(\frac{\sqrt{g}}{\alpha}y\right)\chi'(y)\,dy\right],
        \end{multline*}
        where we have used the fact that $\chi$ is even. One may check that $\chi'$ is strictly negative for $x > 0$ when $\theta \leq 1/2$. This shows that $\eta_2'$ is strictly positive for $x > 0$, and so $\eta_2$ is strictly increasing on $[0,\infty)$ by the mean value theorem.
        
        Before we consider the case $\theta > 1/2$, we prove the asymptotic properties for $\eta_2$ listed in \Crefrange{eq:eta_limit_m_larger_than_one}{eq:eta_limit_m_less_than_one}. These follow by multiplying each side in \Cref{eq:eta_convolution_written_out} with the appropriate factor and taking limits. For instance, suppose that $m>1$, meaning that $\sqrt{g}/\alpha > \pi/h$. For the integral in
        \[
            e^{\pi x/h} \left(e^{-\sqrt{g}x/\alpha} \int_{-\infty}^x e^{\sqrt{g}y/\alpha}\chi(y)\,dy\right) = \frac{\int_{-\infty}^x e^{\sqrt{g}y/\alpha}\chi(y)\,dy}{e^{(\sqrt{g}/\alpha -\pi/h)x}}
        \]
        there are two possibilities: If $\theta=1/2$, then it is possible that the integrand is integrable on the entire real line, meaning that the limit as $x \to \infty$ is zero; otherwise, the integral tends to $\pm \infty$, and so
        \begin{align*}
            \lim_{x \to \infty} \frac{\int_{-\infty}^x e^{\sqrt{g}y/\alpha}\chi(y)\,dy}{e^{(\sqrt{g}/\alpha -\pi/h)x}} &= \lim_{x \to \infty} \frac{e^{\sqrt{g}x/\alpha}\chi(x)}{(\sqrt{g}/\alpha-\pi/h)e^{(\sqrt{g}/\alpha -\pi/h)x}}\\
            &= \frac{1}{\sqrt{g}/\alpha - \pi/h} \frac{\cos(\pi\theta)}{4h^2}
        \end{align*}
        by L'Hôpital's rule. The other limits can be treated in a similar way, with one exception:
        
        The procedure will show that when $m < 1$, we have
        \begin{align*}
            \lim_{x\to \infty} \eta_2(x)e^{\sqrt{g}x/\alpha} &= -\frac{1}{2\alpha\sqrt{g}}\int_{-\infty}^\infty e^{\sqrt{g}y/\alpha}\chi(y)\,dy\notag\\
            &=-\frac{1}{8\pi^2\alpha^2 m}\int_0^\infty y^m \frac{\cos(\pi\theta)y^2+2y+\cos(\pi\theta)}{(y^2+2\cos(\pi\theta)y+1)^2}\,dy\\
            &=-\frac{1}{8\pi^2\alpha^2}\int_0^\infty y^{m-1} \frac{\cos(\pi\theta)y+1}{y^2+2\cos(\pi\theta) y +1}\,dy
        \end{align*}
        where the second and third equality follows from the substitution $y \mapsto e^{\pi y/h}$ and an integration by parts, respectively. The result now follows since the integral on the final line is equal to the right-hand side of \Cref{eq:eta_coefficient_explicit_not_integer} by \Cref{lemma:explicit_expression_for_eta_coefficients}.
        
        Finally, we consider the case of $\theta > 1/2$, which is harder to describe completely, as the integrand in the convolution in \Cref{eq:eta_as_convolution} changes sign. Observe that the claims on the sign of $\eta_2(x)$ for sufficiently large $x$ follows for $m \neq 1/(2\theta)$ from the limits in \Crefrange{eq:eta_limit_m_larger_than_one}{eq:eta_limit_m_less_than_one}. An additional argument is needed for the edge case $m = 1/(2\theta)$, because the limit in \Cref{eq:eta_limit_m_less_than_one} vanishes. It turns out that \Cref{eq:eta_limit_m_larger_than_one} also holds in the special case $m = 1/(2\theta)$, which can be shown with the same method we used to show the other limits. Hence $\eta_2$ is negative for sufficiently large $x$ when $m = 1/(2\theta)$, which exhausts the values of $m$.
    \end{proof}
    \begin{remark}
        It is likely that $\eta_2$ has similar properties to those for the case $\theta \leq 1/2$ when $\theta > 1/2$ and $m \leq 1/(2\theta)$, but we have not been able to prove this.
    \end{remark}
    
    We are now in a position where we can give the sign of $c_3$ in the expansion in \Cref{thm:existence_point_vortex} for $\theta \leq 1/2$.
    
    \begin{proposition}[Sign of $c_3$]
        \label{prop:velocity_sign}
        The constant $c_3$ in \Cref{eq:asymptotic_point_vortex} is negative when $\theta \leq 1/2$. In particular, if $\theta = 1/2$ and $\varepsilon$ is sufficiently small, the waves obtained in \Cref{thm:existence_point_vortex} are left-moving when $\varepsilon > 0$ and right-moving when $\varepsilon < 0$.
    \end{proposition}
    \begin{proof}
        Recall the definition of $\zeta_3$ in \Cref{eq:asymptotic_point_vortex}. From \Cref{theorem:asymptotic_behavior_of_eta} we know that $\eta_2$ is negative, and strictly increasing on $[0,\infty)$. Furthermore, the factor $c_1 + \Phi_y(\cdot,0)$ is positive and strictly decreasing on the same interval. It follows that also $\zeta_3$ is positive and strictly decreasing on $[0,\infty)$.
        
        The harmonic function $H(0)\zeta_3$ on $\R \times (-h, 0)$ assumes the value $0$ at the bottom of the domain and $\zeta_3>0$ at the top of the domain. By the maximum principle, it is positive on the entire domain. Thus we may use the Hopf boundary point lemma (see \cite[Lemma 3.4]{Gilbarg2001}) in order to conclude that $[H(0)\zeta_3]_y(0,-h) >0$. The result will therefore follow if we can show that $[H(0)\zeta_3]_y$ is increasing along the $y$-axis. We will do this by looking at $[H(0)\zeta_3]_x$ on $(0,\infty) \times (-h,0)$. Because of its values on the boundary, it is negative in the interior. Another application of the Hopf boundary point lemma implies that $[H(0)\zeta_3]_{xx}$ is negative on the $y$-axis (except at the point $(0,-h)$, where it vanishes). Since $[H(0)\zeta_3]_{yy} = - [H(0)\zeta_3]_{xx}$ by the harmonicity of $H(0)\zeta_3$, this concludes the proof.
    \end{proof}

    We finish our exposition on a single point vortex with a short discussion on the streamlines of waves obtained in \Cref{thm:existence_point_vortex}. Observe that if $(x(t),y(t))$ denotes the position of a fluid particle at time $t$, then
    \begin{equation}
        \label{eq:particle_before_scaling}
        (\dot{x}(t),\dot{y}(t)) = w(x(t),y(t),t),\\
    \end{equation}
    before the new variables in \Cref{section:traveling_waves}. After introducing the steady variables, \Cref{eq:particle_before_scaling} becomes
    \begin{equation}
        \label{eq:particle_after_scaling}
        (\dot{x}(t),\dot{y}(t)) = w(x(t),y(t)) - (c,0),
    \end{equation}
    meaning that if we only keep the first order terms for $w$ and $c$ from \Cref{thm:existence_point_vortex}, we obtain (keeping the same notation for the paths)
    \begin{equation}
        \label{eq:particle_first_order}
        (\dot{x}(t),\dot{y}(t)) = \varepsilon\nabla^\perp\left(\Phi + c_1 y\right)(x(t),y(t)).
    \end{equation}
    
    \begin{figure}[htb]
        \centering
        \begin{subfigure}[b]{0.495\linewidth}
            \includegraphics{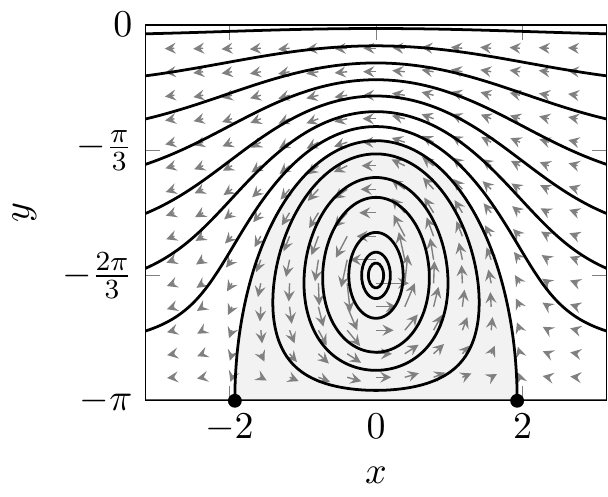}
            \caption{$\theta=1/3$}
            \label{subfig:particle_paths_theta_one_third}
        \end{subfigure}
        \begin{subfigure}[b]{0.495\linewidth}
            \includegraphics{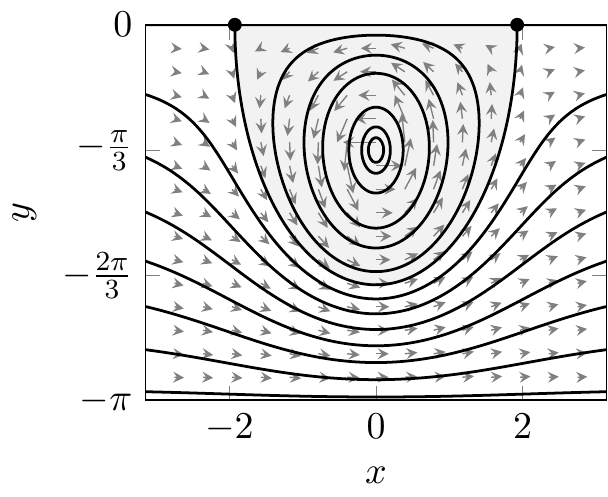}
            \caption{$\theta=2/3$}
            \label{subfig:particle_paths_theta_two_thirds}
        \end{subfigure}
        \caption{Streamlines in the frame of reference traveling with the wave, for $h=\pi$ and $\varepsilon > 0$. The wave corresponding to $\theta=1/3$ propagates to the right, while the wave corresponding to $\theta=2/3$ propagates to the left. The arrows illustrating the vector field on the right hand side of \Cref{eq:particle_first_order} have been scaled here for visibility, and only their direction is quantitatively accurate.}
        \label{fig:particle_paths}
    \end{figure}
    
    We have used this to obtain \Cref{fig:particle_paths}, which shows streamlines in the steady frame moving with the wave. The portraits corresponding to $\theta$ and $1-\theta$ can be obtained from each other by a \SI{180}{\degree} rotation. When $\theta=1/2$, all the streamlines are closed (not shown), so we will focus on the case $\theta \neq 1/2$.  The lines $y=-h$ and $y=0$ are nullclines for the system in \Cref{eq:particle_first_order}, and the points $(x,y)$ with
    \begin{equation}
        \label{eq:phase_portrait_equilibrium_points}
        \begin{aligned}
            x&= \pm h/\pi \arcosh(\abs{2\sin(\pi\theta)\tan(\pi\theta)+\cos(\pi\theta)}),\\ y&=\begin{cases}
                -h & \theta < 1/2\\
                0 & \theta > 1/2 
            \end{cases}
        \end{aligned}
    \end{equation}
    are equilibrium points, corresponding to stagnation points. One may check that
    \[
        h/\pi \arcosh(2\sin(\pi\theta)\tan(\pi\theta)+\cos(\pi\theta)) = \sqrt{3}h\theta + O(\theta^5)
    \]
    as $\theta \downarrow 0$, meaning that the distance between the equilibria is very close to linear in $\theta$ for small $\theta$ (a corresponding statement holds for $1-\theta$ small). They go off to infinity as $\theta \to 1/2$ from either side. The heteroclinic orbit (which can be expressed explicitly in terms of $\arcosh$) connecting the two equilibrium points described in \Cref{eq:phase_portrait_equilibrium_points} encloses a critical layer containing closed streamlines. Outside this region the particles always move in the same direction with respect to the steady frame. This direction is either to the left or right depending on the sign of $\cot(\pi\theta)$ and $\varepsilon$.
    
    We also mention that on infinite depth, the streamlines always look like those in \Cref{subfig:particle_paths_theta_two_thirds}. If the point vortex is situated at $(0,-d)$, the equilibrium points at the surface will be at $(\pm \sqrt{3}d,0)$, and the points on the heteroclinic orbit between these satisfy
    \[
        x^2+(y+d)^2=2dy(1+\coth(y/(2d))),
    \]
    which is close to half an ellipse centered at $(0,-d)$ with semiaxes $\sqrt{3}d$ and $\sim 2.0873 d$. The equilibrium points in \Cref{eq:phase_portrait_equilibrium_points} converge to those on infinite depth as $h \to \infty$ if $d$ is held fixed. 
    
    Because only the first order terms in $\varepsilon$ have been kept in \Cref{eq:particle_first_order}, we do not make any claim about the accuracy of the phase portraits in \Cref{fig:particle_paths} for the full system in \Cref{eq:particle_after_scaling}. That would require further and more thorough analysis, in particular for the case $\theta = 1/2$. Still, the phase portraits can give some indication as to how these waves look beneath the surface. One feature will remain the same for \Cref{eq:particle_after_scaling}: Because of the singularity of $\Phi$ at $(0,-(1-\theta)h)$, the streamlines will always remain closed sufficiently close to the point vortex.
    
\section{Several point vortices}
    \label{section:several_point_vortices}
    We aim to extend the existence result for traveling waves with a single point vortex in \Cref{thm:existence_point_vortex} to a finite number of point vortices on the $y$-axis. As opposed to the single vortex case, where we could choose $\theta$ freely, there will be limitations on the positions that the point vortices can occupy. We will return to this. Suppose that
    \[
        1 > \theta_1 > \theta_2 > \cdots > \theta_n > 0,
    \]
    and that we wish to establish the existence of a traveling wave with point vortex at the points
    \[
        (0,-(1-\theta_1)h), \ldots, (0,-(1-\theta_n)h),
    \]
    the situation being otherwise similar to that of a single point vortex. The admissible surface profiles are those in $\Lambda_{\theta_1}$, as the uppermost point vortex is the most restrictive.
    
    For $\eta \in \Lambda_{\theta_1}$ and $\gamma = (\gamma^1,\ldots,\gamma^n) \in \R^n$ we may define
    \begin{equation}
        \label{eq:phi_gamma}
        \Phi^\gamma \coloneqq \sum_{j=1}^n \gamma^j \Phi^j,
    \end{equation}
    where
    \[
        \Phi^j(x,y) \coloneqq \frac{1}{4\pi}\log\left(\frac{\cosh(\pi x/h) + \cos(\pi(y/h-\theta_j))}{\cosh(\pi x/h) + \cos(\pi(y/h+\theta_j))}\right), \quad j=1,\ldots,n,
    \]
    in $\Omega(\eta)$. We will seek solutions of the form
    \[
        w = \nabla^\perp[H(\eta)\zeta + \Phi^\gamma],
    \]
    cf. \Cref{eq:velocity_fields_considered_point_vortex} for a single point vortex.
    
    The main difference from the single point vortex case is of course the vorticity equation, \Cref{eq:vorticity_transport_equation_weak_steady}, which needs to be imposed for each of the point vortices. For the $i$th point vortex, the vorticity equation reduces to
    %Use \cr instead of // in substack because of weird interaction with multlined
    \[
        (c,0)\begin{multlined}[t]=\nabla^\perp[H(\eta)\zeta](0,-(1-\theta_i)h)\\+ \frac{1}{4h}\biggl(\gamma^i\cot(\pi\theta_i) + \sum_{\substack{j=1 \cr j\neq i}}^n \gamma^j\left[\cot\left(\pi\frac{\theta_i + \theta_j}{2}\right)-\cot\left(\pi\frac{\theta_i-\theta_j}{2}\right)\right],0\biggr),
        \end{multlined}
    \]
    which, if we assume that $\eta$ and $\zeta$ are even (see the discussion before \Cref{eq:vorticity_equation_point_vortex}), can be written more succinctly as
    \begin{equation}
        \label{eq:vorticity_equation_multiple_vortices}
        c\mathbf{1} = -([H(\eta)\zeta]_y(0,-(1-\theta_i)h))_{i=1}^n + \Theta\gamma.
    \end{equation}
    Here, we have defined $\mathbf{1} \coloneqq (1,\ldots,1) \in \R^n$ and the matrix $\Theta \in \R^{n \times n}$ by
    \begin{equation}
        \label{eq:theta_matrix_definition}
        \Theta_{i,j} = \begin{dcases}
            \frac{1}{4h}\cot(\pi \theta_i) & i = j\\
            \frac{1}{4h}\left(\cot\left(\pi\frac{\theta_i + \theta_j}{2}\right)-\cot\left(\pi\frac{\theta_i-\theta_j}{2}\right)\right) & i \neq j
        \end{dcases}
    \end{equation}
    for $1 \leq i,j \leq n$.
    
    As opposed to for one vortex, it is now more natural to use the wave velocity $c$ as the bifurcation parameter. We will therefore write $\varepsilon$ instead of $c$ in order to have notation that is more consistent with the one vortex case. The idea is to use the vortex strengths $\gamma$ in order to balance \Cref{eq:vorticity_equation_multiple_vortices}, which is possible when $\Theta$ is invertible. It should be emphasized that this is almost always the case (\Cref{thm:point_vortex_configurations_measure_zero}), but that there always are configurations of $n$ point vortices that yield singular $\Theta$ (\Cref{prop:always_bad_configurations}). We have already seen such a configuration, albeit a trivial one: For the case $n=1$ one has $\Theta=0$ when $\theta = 1/2$.
    
    We make the necessary redefinitions
    \begin{align*}
        X^s &\coloneqq H_\text{even}^s(\R) \times H_\text{even}^s(\R) \times \R^n,\\
        Y^s &\coloneqq H_\text{even}^{s-2}(\R) \times  H_\text{even}^s(\R) \times \R^n,\\
        U_{\theta_1}^s &\coloneqq \left\{(\eta,\zeta, \gamma) \in X^s : \eta \in \Lambda_{\theta_1}\right\},
    \end{align*}
    and proceed to define, for $s > 3/2$, the map  $\map{F_1}{U_{\theta_1}^s \times \R}{H_\text{even}^{s-2}(\R)}$ by
    \begin{multline*}
        F_1(\eta, \zeta, \gamma, \varepsilon) = \varepsilon\left[\frac{\eta'\zeta'+G(\eta)\zeta}{\jb{\eta'}^2}+\Phi_y^\gamma\right]\\ +\frac{(\zeta' + (1,\eta') \cdot \nabla \Phi^\gamma)^2 +(G(\eta)\zeta +(-\eta',1)\cdot \nabla \Phi^\gamma)^2}{2\jb{\eta'}^2}
            + g\eta -\alpha^2\kappa(\eta),
    \end{multline*}
    the map $\map{F_2}{U_{\theta_1}^s \times \R}{H_\text{even}^s(\R)}$ by
    \[
        F_2(\eta,\zeta,\gamma,\varepsilon) \coloneqq \varepsilon\eta + \zeta+\Phi^\gamma,
    \]
    and finally the map $\map{F_3}{U_{\theta_1}^s \times \R}{\R^n}$ by
    \[
        F_3(\eta,\zeta,\gamma,\varepsilon)\coloneqq \Theta\gamma - \varepsilon\mathbf{1} - ([H(\eta)\zeta]_y(0,-(1-\theta_i)h))_{i=1}^n.
    \]
    In all of these definitions, the function $\Phi^\gamma$ and its derivatives are evaluated at $(x,\eta(x))$, which is suppressed for readability.
    
    We now define $F \coloneqq \map{(F_1, F_2, F_3)}{U_\theta^s}{Y^s}$, and seek solutions of the equation
    \begin{equation}
        \label{eq:zcs_formulation_several_point_vortices}
        F(\eta,\zeta,\gamma,\varepsilon)=0,
    \end{equation}
    which has the origin as a trivial solution. We are led to the following analog of \Cref{thm:existence_point_vortex} for several point vortices, establishing the existence of a family of small, localized solutions, assuming that $\Theta$ is nonsingular. The resulting waves have one critical layer for each point vortex, assuming that no component of $\gamma$ vanishes.
    
    \begin{theorem}[Traveling waves with several point vortices]
        \label{thm:existence_several_point_vortices}
        Let $s > 3/2$, and let $1 > \theta_1 > \theta_2 > \cdots > \theta_n > 0$. Suppose that the matrix $\Theta$ defined in \Cref{eq:theta_matrix_definition} is invertible. Then there exists an open interval $I \ni 0$ and a $C^\infty$-curve
        \[
            \arraycolsep=0.03\textwidth
            \begin{array}{ccc}
                I & \to & (H_\text{even}^s(\R) \cap \Lambda_{\theta_1}) \times H_\text{even}^s(\R) \times \R^n \times \R\\
                \varepsilon & \mapsto & (\eta(\varepsilon),\zeta(\varepsilon),\gamma(\varepsilon), \varepsilon)
            \end{array}
        \]
        of solutions with velocity $c=\varepsilon$ to the Zakharov--Craig--Sulem formulation, \Cref{eq:zcs_formulation_several_point_vortices}, for point vortices of strengths $\gamma^1(\varepsilon),\ldots, \gamma^n(\varepsilon)$ situated at
        \[
            (0,-(1-\theta_1)h), \ldots , (0,-(1-\theta_n)h).
        \]
        The solutions fulfil
        \begin{equation}
            \label{eq:multiple_point_vortices_asymptotic}
            \begin{aligned}
                \eta(\varepsilon)&= \eta_2\varepsilon^2 + O(\varepsilon^4),\\
                \zeta(\varepsilon)&= \zeta_3 \varepsilon^3 + O(\varepsilon^4),\\
                \gamma(\varepsilon)&= \gamma_1 \varepsilon + \gamma_3 \varepsilon^3 + O(\varepsilon^4),
            \end{aligned}
        \end{equation}
        in their respective spaces as $\varepsilon \to 0$, where $\gamma_1 \coloneqq \Theta^{-1}\mathbf{1}$, the function $\eta_2 \in H_\text{even}^s(\R)$ is defined by
        \begin{align*}
            \eta_2 \coloneqq -(g-\alpha^2\partial_x^2)^{-1}\chi, \quad \chi \coloneqq \Phi_y^{\gamma_1}(\cdot,0) + \frac{1}{2}\Phi_y^{\gamma_1}(\cdot,0)^2,
        \end{align*}
        and where
        \begin{align*}
            \zeta_3 &= - \eta_2(1+\Phi_y^{\gamma_1}(\cdot,0)),\\
            \gamma_3 &= \Theta^{-1}([H(0)\zeta_3]_y(0,-(1-\theta_i)h))_{i=1}^n,
        \end{align*}
        with $\Phi^{\gamma_1}$ as in \Cref{eq:phi_gamma} and $H$ as in \Cref{def:harmonic_extension_operator}.
        
        Moreover, there is a neighborhood of the origin in $U_\theta^s \times \R$ such that this curve describes all solutions to $F(\eta, \zeta, \gamma,\varepsilon)=0$ in that neighborhood.
    \end{theorem}
    \begin{proof}
        As for a single point vortex, we wish to apply the implicit function theorem at the origin. We find the derivative
        \[
            D_X F(0,0,0,0) = \begin{bmatrix}
                g - \alpha^2 \partial_x^2 & 0 & 0\\
                0 & I_{H_\text{even}^s(\R)} & 0\\
                0 & -([H(0)\cdot]_y(0,-(1-\theta_i)h))_{i=1}^n & \Theta
            \end{bmatrix},
        \]
        where $([H(0)\cdot]_y(0,-(1-\theta_i)h))_{i=1}^n$ means the operator $H_\text{even}^s(\R) \to \R^n$ defined by
        \[
            \zeta \mapsto ([H(0)\zeta]_y(0,-(1-\theta_i)h))_{i=1}^n.
        \]
        Recalling that $g-\alpha^2\partial_x^2$ and $\Theta$ are invertible by the discussion after \Cref{eq:point_vortex_map_derivative_X} and by assumption, respectively, $D_XF(0,0,0,0)$ is an isomorphism.
        
        Hence we can use the implicit function theorem to deduce the existence of an open interval $I$ around zero, an open set $V \subseteq U_{\theta_1}^s$ containing the origin, and a map $f \in C^\infty(I,V)$ such that for $(\eta,\zeta,\gamma,\varepsilon) \in V \times I$, we have
        \[
            F(\eta,\zeta,\gamma,\varepsilon) = 0 \iff (\eta,\zeta,\gamma) = f(\varepsilon).
        \]
        The terms in the expansion in \Cref{eq:multiple_point_vortices_asymptotic} can be obtained as in the proof of \Cref{thm:existence_point_vortex}.
    \end{proof}
    \begin{remark}
        It is worth mentioning that on infinite depth, the matrix $\Theta$ is always invertible. A corresponding existence theorem for infinite depth would thus hold for any configuration.
    \end{remark}
    \begin{remark}
        An extension of the existence result in \Cref{thm:existence_several_point_vortices} to point vortices that are not all on the same vertical line would require a different argument than the one we have used. The main issue is that assuming $\eta$ and $\zeta$ to be even is then no longer sufficient to satisfy the vertical component of the vorticity equation, like we did to obtain \Cref{eq:vorticity_equation_multiple_vortices}.
    \end{remark}
    
    One may note that the sign reversal of the wave velocity about the midpoint $\theta=1/2$ that we saw with the single point vortex, can be seen also for several point vortices, albeit in a different manner. If the matrix $\Theta$ corresponds to $1 > \theta_1 > \cdots > \theta_n > 0$, and we reflect the vortices across the line $y=-h/2$ by considering $\vartheta_i \coloneqq 1 - \theta_i, 1 \leq i \leq n$ instead (without reordering them), then the new matrix is $-\Theta$. This causes a swap of sign on the leading order vortex strengths, $\gamma_1 = \Theta^{-1}\mathbf{1}$.
    
    We have pointed out that the matrix $\Theta$ is not invertible for all configurations of point vortices, and gave the trivial example of $\theta = 1/2$ for a single point vortex. This example, together with \Cref{thm:existence_point_vortex}, also shows that invertibility of $\Theta$ is not a necessary condition for the existence of a traveling wave with point vortices in those points. See also \Cref{remark:existence_when_det_theta_vanishes}.
    
    The only case for multiple point vortices on the $y$-axis where we can feasibly describe the admissible positions directly is for $n=2$. In fact, we give a complete description of when $\Theta$ is invertible in \Cref{prop:two_point_vortices}; see also \Cref{fig:two_point_vortices_positions}, which presents this result graphically. One may observe that the midpoint between the bottom and surface plays a role also here.
    
    \begin{figure}[htb]
        \centering
        \includegraphics{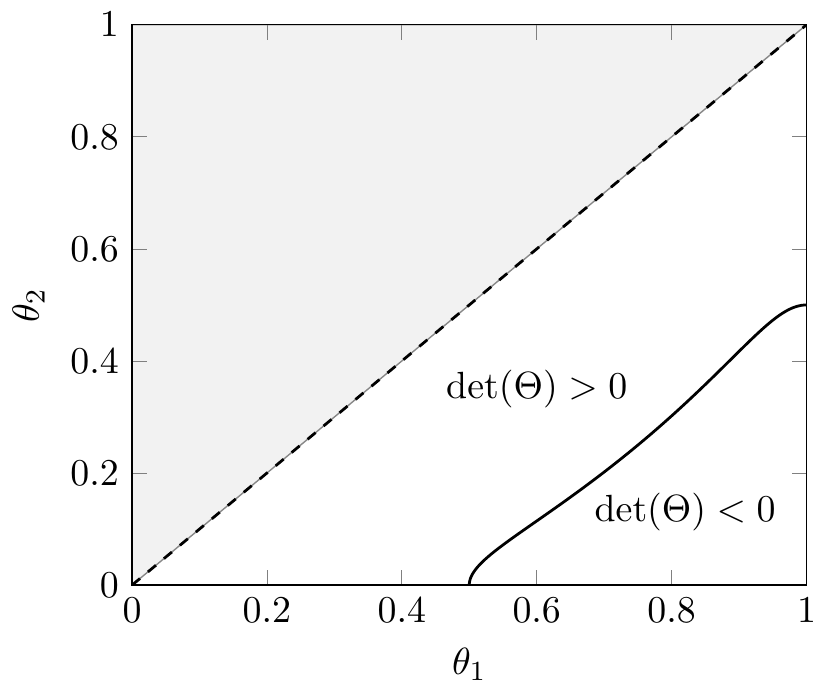}
        \caption{The determinant of $\Theta$ for the case $n=2$ as a function of $(\theta_1, \theta_2)$. The determinant vanishes along the solid black curve, which is given explicitly as a parametrization in \Cref{prop:two_point_vortices}. (In the figure, the level curve for $\det(\Theta)=0$ is computed numerically.)}
        \label{fig:two_point_vortices_positions}
    \end{figure}
    
    \begin{proposition}[$\Theta$ for $n=2$]
        \label{prop:two_point_vortices}
        For two point vortices, we have the following:
        \begin{enumerate}[(i)]
            \item If $\theta_1 \leq 1/2$, then $\Theta$ is invertible for all $\theta_2 \in (0,\theta_1)$.
            \item If $\theta_1 > 1/2$, then $\Theta$ is invertible for all $\theta_2 \in (0,\theta_1)$ except for exactly one value, $0<\hat{\theta}_2(\theta_1)<1/2$. The graph of $\map{\hat{\theta}_2}{\left(1/2,1\right)}{\left(0,1/2\right)}$ is described by the curve
            \[
                \arraycolsep=0.03\textwidth
                \begin{array}{ccc}
                    \left(\pi/4,3\pi/4\right) & \to & \left(1/2,1\right) \times \left(0,1/2\right)\\
                    t & \mapsto & (t+f(t),t-f(t))/\pi
                \end{array}
            \]
            where $\map{f}{\left(\pi/4,3\pi/4 \right)}{\R}$ is defined by
            \[
                f(x) \coloneqq \arccot\left(\sqrt{\frac{1}{2}\left(\cot(x)^2+\sqrt{4-3\cot(x)^4}\right)}\right).
            \]
        \end{enumerate}
    \end{proposition}
    \begin{proof}
        It is useful to write the determinant of $\Theta$ as
        \[
            \det(\Theta) = \frac{1}{16h^2}\left[\cot(\pi\theta_1)\cot(\pi\theta_2)+\frac{4\sin(\pi\theta_1)\sin(\pi\theta_2)}{(\cos(\pi\theta_2)-\cos(\pi\theta_1))^2}\right].
        \]
        One immediately observes that the second term inside the parentheses is always strictly positive. If $\theta_1 \leq 1/2$, then one has in addition that the first term is nonnegative for any $\theta_2 \in (0,\theta_1) \subseteq (0,1/2)$. This proves the first part of the proposition.
        
        For the second part, let us first prove that there is exactly one value of $\theta_2$ for each $\theta_1 \in (1/2,1)$ that makes $\Theta$ singular, and that this value lies in the interval $(0,1/2)$. For fixed $\theta_1 \in (1/2,1)$ the determinant is strictly increasing in $\theta_2$, and tends to $-\infty$ as $\theta_2 \downarrow 0$, and to $\infty$ as $\theta_2 \uparrow \theta_1$. Hence it vanishes at exactly one value of $\theta_2$, say $\hat{\theta_2}(\theta_1)$. Because the determinant is positive when $\theta_2 = 1/2$, this value must necessarily lie in the interval $(0,1/2)$.
        
        We now move to the parametrization of the graph of the map $\map{\hat{\theta}_2}{(1/2,1)}{(0,1/2)}$. One may note from \Cref{fig:two_point_vortices_positions} that there is symmetry across the diagonal line
        \[
            \{(\theta_1,\theta_2) \in (0,1)^2 : \theta_1 > \theta_2, \theta_1 + \theta_2 = 1\},
        \]
        which suggests making a change of variables. By letting
        \begin{equation}
            \label{eq:two_vortices_determinant_change_of_variables}
            \phi_1 \coloneqq \pi\frac{\theta_1 + \theta_2}{2},\quad \phi_2 \coloneqq \pi\frac{\theta_1-\theta_2}{2},
        \end{equation}
        we can write the determinant in the form
        \[
            \det(\Theta) = \frac{1}{16h^2}\left[\frac{\cot(\phi_1)^2\cot(\phi_2)^2-1}{\cot(\phi_2)^2-\cot(\phi_1)^2}+\cot(\phi_2)^2-\cot(\phi_1)^2\right],
        \]
        which leads us to solve the quadratic equation
        \[
            x^2-ax+a^2-1=0,\quad a\coloneqq \cot^2(\phi_1), \quad x\coloneqq \cot^2(\phi_2)
        \]
        for $x$, given $a$. Doing this yields the parametrization, by using $\phi_1$ as the parameter (some care has to be taken to ensure that one picks the right branches of the functions involved) and going back to the original variables by inverting \Cref{eq:two_vortices_determinant_change_of_variables}.
    \end{proof}
    
    \begin{remark}
        \label{remark:existence_when_det_theta_vanishes}
        By employing the parametrization of the graph of $\hat{\theta}_2$ provided by \Cref{prop:two_point_vortices}, one can show that the each column of $\Theta$ is linearly independent from $\mathbf{1}$ when $\det(\Theta)=0$. This implies that an argument similar to that of \Cref{thm:existence_several_point_vortices} can be performed, by using the vortex strength $\gamma^1$ as the bifurcation parameter, instead of $c$. Thus it is possible to show existence for any configuration when $n=2$. An extension of this argument to $n > 2$ is harder, because it requires the rank of $\Theta$ to be $n-1$.
    \end{remark}
    
    While the set of configurations that make $\det(\Theta)$ vanish is hard to describe in general when $n > 2$, some observations can be made. Of course, if $n \geq 2$, and as long as the derivative of $\det(\Theta)$ with respect to the variable $(\theta_1,\ldots,\theta_n)$ does not vanish at a point where $\det(\Theta) = 0$, the zero set of $\det(\Theta)$ is locally a smooth manifold of dimension $n-1$ around that point by the implicit function theorem. When $n=2$, the zero set is actually the graph of a smooth function in $\theta_1$ by \Cref{prop:two_point_vortices}, and numerical evidence suggests that the zero set is the graph of a smooth function in $(\theta_1,\theta_2)$ when $n=3$. Actually checking that the derivative does not vanish is hard, but we have the following theorem:
    
    \begin{theorem}
        \label{thm:point_vortex_configurations_measure_zero}
        The subset of configurations of point vortices in
        \[\{(\theta_1,\ldots,\theta_n) \in (0,1)^n : 1 > \theta_1 > \theta_2 > \cdots > \theta_n > 0\}
        \]
        such that $\Theta$ is not invertible has measure zero.
    \end{theorem}
    \begin{proof}
        Each entry in $\Theta$ is analytic in each $\theta_i$ for $\theta_1,\ldots,\theta_{i-1},\theta_{i+1},\ldots,\theta_n$ fixed. It follows that $\det(\Theta)$ also has this property, when viewed as a function
        \[
            U \coloneqq \{(\theta_1,\ldots,\theta_n) \in (0,1)^n : 1 > \theta_1 > \theta_2 > \cdots > \theta_n > 0\} \to \R.
        \]

        We first verify that $\det(\Theta)$ does not vanish identically on $U$. To that end, fix $1/2 > \tilde{\theta}_1 > \tilde{\theta}_2 > \cdots > \tilde{\theta}_n > 0$ and consider $\theta_1 = \varepsilon \tilde{\theta}_1,\ldots \theta_n = \varepsilon \tilde{\theta}_n$ for $1 > \varepsilon > 0$. The purpose of the upper bound of $1/2$ is to make sure that $\tan(\pi\theta_i)$ is well defined for all $1 \leq i \leq n$. Observe now that if we let $T\coloneqq\diag(\tan(\pi \theta_k))_{k=1}^n$, then
        \[
            4h[T\Theta]_{i,j} = \begin{cases}
                1 & i = j\\
                \tan(\pi\theta_i)\left(\cot\left(\pi\frac{\theta_i + \theta_j}{2}\right) - \cot\left(\pi\frac{\theta_i - \theta_j}{2}\right) \right) & i \neq j
            \end{cases},
        \]
        where
        \[
            \lim_{\varepsilon \downarrow 0} \tan(\varepsilon\pi\tilde{\theta}_i)\left(\cot\left(\varepsilon\pi\frac{\tilde{\theta}_i + \tilde{\theta}_j}{2}\right) - \cot\left(\varepsilon\pi\frac{\tilde{\theta}_i - \tilde{\theta}_j}{2}\right) \right)= -\frac{4\tilde{\theta}_i \tilde{\theta}_j}{\tilde{\theta}_i^2 - \tilde{\theta}_j^2}
        \]
        for $i \neq j$. It follows that $\diag(\tan(\pi \theta_k))_{k=1}^n\Theta$ has a limit in $B(\R^n)$ as $\varepsilon \downarrow 0$, and that this limit is
        \begin{equation}
            \label{eq:diag_theta_limit_pos}
            \lim_{\varepsilon \downarrow 0} \diag(\tan(\pi \theta_k))_{k=1}^n\Theta = \frac{1}{4h}(I_{\R^n} - B),
        \end{equation}
        where we have defined $B \in \R^{n \times n}$ by
        \begin{equation}
            \label{eq:def_of_b_matrix}
            B_{i,j} \coloneqq \begin{cases}
                0 & i = j\\
                4\tilde{\theta}_i \tilde{\theta}_j(\tilde{\theta}_i^2 - \tilde{\theta}_j^2)^{-1} & i \neq j
            \end{cases}.
        \end{equation}
        In particular, $B$ is skew-symmetric, which implies that $I_{\R^n}-B$ is invertible. Since the set of invertible operators is open, so is the matrix $\diag(\tan(\pi \theta_k))_{k=1}^n\Theta$ for sufficiently small $\varepsilon$, which in turn means that $\Theta$ is invertible for such $\varepsilon$.
        
        Finally, the set $U$ is connected. Hence, since we know that $\det(\Theta)$ is analytic in each variable and does not vanish identically, we infer\footnote{This follows by induction on the dimension, by using the well known result in one dimension.} that the subset of $U$ on which $\det(\Theta)$ vanishes has measure zero.
    \end{proof}
    
    In general we cannot do better than \Cref{thm:point_vortex_configurations_measure_zero}, in the sense that for any $n \geq 1$ there will always be a configuration of $n$ point vortices that makes $\det(\Theta)$ vanish.
    
    \begin{proposition}
        \label{prop:always_bad_configurations}
        There are always configurations of point vortices in
        \[
            \{(\theta_1,\ldots,\theta_n) \in (0,1)^n : 1 > \theta_1 > \theta_2 > \cdots > \theta_n > 0\}
        \]
        where $\Theta$ is singular.
    \end{proposition}
    \begin{proof}
        The matrix appearing on the right-hand side of \Cref{eq:diag_theta_limit_pos} in the proof of \Cref{thm:point_vortex_configurations_measure_zero} has a positive determinant. Indeed, the matrix $B$ defined in \Cref{eq:def_of_b_matrix} is skew-symmetric, so its spectrum is purely imaginary. Moreover, since the matrix is real, the eigenvalues are either zero or appear in complex conjugate pairs.
        
        Say that the first $m$ eigenvalues $\lambda_1,\ldots,\lambda_m$ of $B$ are zero and that
        \[
            \lambda_{m+2j-1} = \overline{\lambda_{m+2j}}=i\mu_j,\quad j=1,\ldots,(n-m)/2,
        \]
        where the $\mu_j$ are real. Then it follows that
        \begin{align*}
            \det\left(\frac{1}{4h}(I_{\R^n}-B)\right) &= \frac{1}{(4h)^n} \det(I_{\R^n}-B)\\
            &=\frac{1}{(4h)^n}(1+\mu_1^2)(1+\mu_2^2)\cdots(1+\mu_{(n-m)/2}^2),
        \end{align*}
        because the determinant of a matrix is equal to the product of its eigenvalues (taking algebraic multiplicity into account). By \Cref{eq:diag_theta_limit_pos} we then have
        \begin{equation}
            \label{eq:det_product_positive}
            \det(\Theta) \prod_{k=1}^n \tan(\pi\theta_k) > 0
        \end{equation}
        for small $\varepsilon > 0$ (as in the proof of \Cref{thm:point_vortex_configurations_measure_zero}) by continuity of the determinant. Since all the tangents are also positive, this implies that $\det(\Theta) > 0$ for small $\varepsilon > 0$.
        
        It remains to exhibit a configuration where $\det(\Theta) < 0$. To that end, fix $\frac{1}{2} > \tilde{\theta}_1 > \tilde{\theta}_2 > \tilde{\theta}_3 > \cdots > \tilde{\theta}_n > 0$ and consider $\theta_1 = 1- \varepsilon \tilde{\theta}_1, \theta_2 = \varepsilon \tilde{\theta}_2,\ldots \theta_n = \varepsilon \tilde{\theta}_n$ for $1 > \varepsilon > 0$. Proceeding as in the proof of \Cref{thm:point_vortex_configurations_measure_zero} we find
        \begin{equation}
            \label{eq:diag_theta_limit_neg}
            \lim_{\varepsilon \downarrow 0} \diag(\tan(\pi \theta_k))_{k=1}^n\Theta = \frac{1}{4h}(I_{\R^n} - \tilde{B}),
        \end{equation}
        where we have defined $\tilde{B} \in \R^{n \times n}$ by
        \[
            \tilde{B}_{i,j} \coloneqq \begin{cases}
                0 & \text{$i=j$ or $i=1$ or $j=1$},\\
                4\tilde{\theta}_i \tilde{\theta}_j(\tilde{\theta}_i^2 - \tilde{\theta}_j^2)^{-1} & \text{otherwise}.
            \end{cases}
        \]
        This matrix is still skew-symmetric like $B$, and so the right-hand side of \Cref{eq:diag_theta_limit_neg} has a positive determinant, as before. Hence \Cref{eq:det_product_positive} holds for small $\varepsilon$. However, now $\tan(\pi\theta_1)$ is negative and the rest of the tangents are positive, meaning that $\det(\Theta)$ must be negative.
    \end{proof}
\section{Explicit expressions for infinite depth}
    \label{section:infinite_depth}
    In this section we give some explicit expressions for periodic waves with a point vortex on infinite depth, constructed in \cite{Shatah2013}. We will adopt the notation and conventions used there. The fluid domain for the trivial surface is $\R \times (-\infty,1)$ and the waves have period $2\pi L$. The stream function for the rotational part is denoted by $\mathbf{G}$.
    
    \begin{proposition}[Stream function]
        \label{prop:stream_function_periodic}
        The stream function for the rotational part is given by
        \[
            \mathbf{G}(x,y) = \frac{1}{4\pi} \log\left(\frac{\cos(x/L)-\cosh(y/L)}{\cos(x/L)-\cosh((y-2)/L)}\right).
        \]
    \end{proposition}
    \begin{proof}
        We wish to find the stream function $\map{\mathbf{G}}{\R \times (-\infty,1)}{\R}$ corresponding to equally spaced point vortices of unit strength at the points $2\pi L\Z \times \{0\} \subseteq \R^2$, and which is such that this stream function vanishes at the surface, $\R \times \{1\}$. By symmetry, it must be the case that $\mathbf{G}_x$ vanishes on $\pi L(1 + 2\Z) \times (-\infty,1)$. This leads us to the boundary value problem
        \[
            \Delta \mathbf{G} = \delta, \quad \restr{\mathbf{G}}{y=0} = 0, \quad \restr{\mathbf{G}_x}{x=\pm \pi L} = 0,
        \]
        on $(-\pi L,\pi L) \times (-\infty,1)$. This equation can be dealt with using \Cref{thm:greens_functions_mixed} in \Cref{appendix}.
        
        In order to apply \Cref{thm:greens_functions_mixed} we require a conformal map satisfying the requirements in the theorem statement. One may check (see \cite[Sections 7.1 and 7.2]{Varholm2014}) that
        \begin{equation}
            \label{eq:infinite_depth_conformal_map}
            f(z) \coloneqq \frac{\tanh(1/(2L)) -\tanh((1+iz)/(2L))}{\tanh(1/(2L))+\tanh((1+iz)/(2L))}
        \end{equation}
        defines a bijective conformal map from the half strip $(-\pi L, \pi L) \times (-\infty,1)$ onto the slit unit disk $\mathbb{D} \setminus ((0,\exp(-1/L)) \times \{0\})$, and which is such that
        \begin{enumerate}[(i)]
            \item The origin is fixed.
            \item The surface is mapped to the unit circle.
            \item The sides $\{\pm\pi L\} \times (-\infty,1)$ are mapped to the slit.
        \end{enumerate}
        The result now follows by taking the logarithm of the modulus of the map $f$ in \Cref{eq:infinite_depth_conformal_map}.
    \end{proof}
    
    By using \Cref{prop:stream_function_periodic}, we can obtain an explicit expression for the leading order wave velocity $c_1$, and a Fourier series for the leading order surface profile $\eta_*$:
    
    \begin{proposition}[$c_1$ and $\eta_*$]
        The leading-order wave velocity $c_1$ and surface profile $\eta_*$ are given by
        \begin{align*}
            c_1 &= -\frac{1}{4\pi L}\coth(1/L),\\
            \eta_* &= - \frac{1}{4\pi^2} \sum_{n=1}^\infty \frac{n}{gL^2 + \alpha^2 n^2} e^{-n/L}\cos(nx/L),
        \end{align*}
        respectively.
    \end{proposition}
    \begin{proof}
        Recall how the wave velocity appeared on the right-hand side of \Cref{eq:phi_gradient}. By using the final part of \Cref{thm:greens_functions}, we find
        \[
            c_1 =\frac{i}{4\pi}\overline{\left(\frac{f''(0)}{f'(0)}\right)} = - \frac{1}{4\pi L}\coth(1/L),
        \]
        where $f$ is the conformal map introduced in \Cref{eq:infinite_depth_conformal_map} in the proof of \Cref{prop:stream_function_periodic}.
        
        We now move to the surface profile. From \cite{Shatah2013} we know that
        \begin{equation}
            \label{eq:leading_order_surface_infinite_depth}
            \eta_* = -(g-\alpha^2\partial_x^2)^{-1}\left(\chi - \frac{1}{2\pi L}\int_{-\pi L}^{\pi L} \chi\,d\mu\right),
        \end{equation}
        where $\chi$ is defined by
        \[
            \chi(x) \coloneqq c_1 \mathbf{G}_y(x,1) + \frac{1}{2}\mathbf{G}_y(x,1)^2.
        \]
        Written out, we have
        \[
            \chi(x) =\frac{1}{8\pi^2 L^2} \frac{\cosh(1/L)\cos(x/L)-1}{(\cos(x/L)-\cosh(1/L))^2}
        \]
        with the elementary antiderivative
        \[
            \ad{\chi}{1}(x)=-\frac{1}{8\pi^2 L} \frac{\sin(x/L)}{\cos(x/L)-\cosh(1/L)}.
        \]
        In particular, this means that
        \[
            \int_{-\pi L}^{\pi L} \chi\,d\mu = \ad{\chi}{1}(\pi L)-\ad{\chi}{1}(-\pi L)=0,
        \]
        so that \Cref{eq:leading_order_surface_infinite_depth} reduces to
        \begin{equation}
            \label{eq:leading_order_surface_infinite_depth_reduced}
            \eta_* = -(g-\alpha^2\partial_x^2)^{-1}\chi.
        \end{equation}
        
        In order to find the Fourier series for $\eta_*$, we require the Fourier series of $\chi$. We may write
        \[
            \ad{\chi}{1}(x)=\frac{1}{i}\left[1+\frac{e^{-ix/L-1/L}}{1-e^{-ix/L-1/L}} - \frac{1}{1-e^{ix/L-1/L}}\right],
        \]
        which, by expanding into geometric series, means that
        \[
            \ad{\chi}{1}(x) = i\sum_{n=-\infty}^\infty \sgn(n)e^{-\abs{n}/L}e^{inx/L}.
        \]
        Hence, by termwise differentiation, we obtain
        \[
            \chi(x) = \frac{1}{4\pi^2 L^2}\sum_{n=1}^\infty n e^{-n/L} \cos(n x /L),
        \]
        which, combined with \Cref{eq:leading_order_surface_infinite_depth_reduced}, yields the result.
    \end{proof}
    
    One may note that
    \[
        c_1 = -\frac{1}{4\pi} + O(1/L^2)
    \]
    as $L \to \infty$, which agrees with the speed of the solitary waves on infinite depth. When $L$ is large, the surface profile is very similar to the surface in the localized case, see \Cref{subfig:periodic_surface_long_waves}. At the other extreme, the first terms in the Fourier series will dominate.
    
    \begin{figure}[htb]
        \centering
        \begin{subfigure}[b]{0.495\linewidth}
            \includegraphics{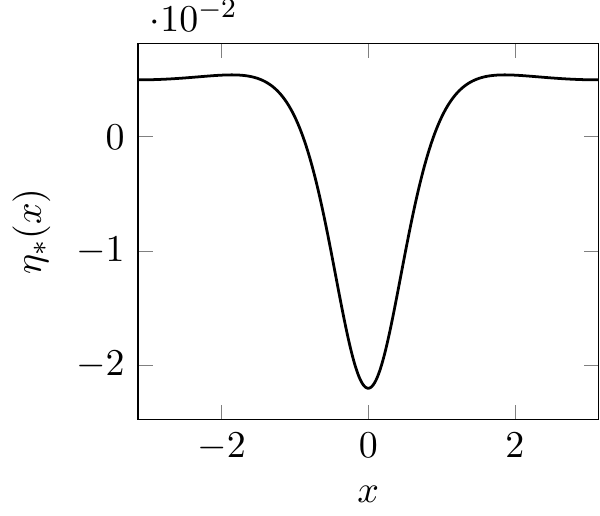}
            \caption{$L=1$ (Whole period)}
            \label{subfig:periodic_surface_short_waves}
        \end{subfigure}
        \begin{subfigure}[b]{0.495\linewidth}
            \includegraphics{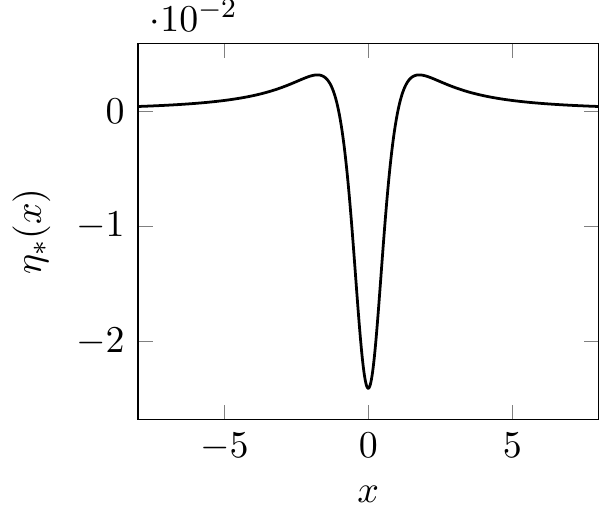}
            \caption{$L=100$ (Only part of a period)}
            \label{subfig:periodic_surface_long_waves}
        \end{subfigure}
        \caption{The leading order surface profile term, $\eta_*$, when $g=1, \alpha^2=0.01$, cf. \cite[Figure 1]{Shatah2013}.}
        \label{fig:periodic_surface}
    \end{figure}
    
\appendix
\section{}
    \label{appendix}
    In this appendix, we provide two theorems that are used in order to get exact expressions for the rotational part of the stream function. Except for the final part, \Cref{thm:greens_functions} is a standard result \cite[p. 166]{Markushevich1965a}. \Cref{thm:greens_functions_mixed} is a less well known extension of \Cref{thm:greens_functions}.
    \begin{theorem}[Green's functions in $\R^2$]
        \label{thm:greens_functions}
        Suppose that $\Omega \subsetneq \R^2$ is a simply connected domain and that $z_0 \in \Omega$. Furthermore, suppose that $\map{f}{\Omega}{\mathbb{D}}$ is a bijective conformal map onto the open unit disk, extending continuously to a function $\overline{\Omega} \to \overline{\mathbb{D}}$ and satisfying $f(z_0)=0$. Then the function $\map{\varphi}{\Omega}{\R}$ defined by
        \[
            \varphi(z) \coloneqq \frac{1}{2\pi}\log(|f(z)|)
        \]
        is in $L_\textnormal{loc}^1(\Omega)$, extends continuously to the boundary of $\Omega$, and satisfies
        \begin{align*}
            \Delta \varphi &= \delta_{z_0},\\
            \restr{\varphi}{\partial \Omega} &= 0.
        \end{align*}
        Furthermore, the harmonic function $h$ defined by
        \[
            h(z) \coloneqq \varphi(z) - \frac{1}{2\pi} \log{(|z-z_0|)}
        \]
        satisfies
        \[
            \nabla h (z_0) = \frac{1}{4\pi} \overline{\left(\frac{f''(z_0)}{f'(z_0)}\right)}
        \]
        after identifying $\R^2$ and $\C$ via $(x,y) \mapsto x+iy$.
    \end{theorem}
    \begin{proof}
        We first check the boundary values of the function $\varphi$. By assumption, $f$ extends continuously to $\partial \Omega$, and every point on $\partial \Omega$ must necessarily be mapped to the unit circle. It is thus immediate that $\varphi$ also extends continuosly to the boundary, and moreover, vanishes there.
        
        Identify now $\R^2$ and $\C$. Observe that since $f(z_0)=0$, we have
        \[
            f(z)=g(z)(z-z_0),\quad z \in \Omega
        \]
        for some holomorphic function $g$, where $|g|>0$. Indeed, we must have $g(z_0)=f'(z_0) \neq 0$ because $f$ is injective, and the injectivity of $f$ also ensures that there can be no other roots. Thus
        \[
            \varphi(z) = \frac{1}{2\pi}\log{(\abs{z-z_0})} + h(z),
        \]
        where
        \[
            h(z) \coloneqq \frac{1}{2\pi}\re\log{(g(z))}
        \]
        is harmonic by $|g|>0$ and the Cauchy-Riemann equations. Hence, by \Cref{prop:delta_vorticity}, the function $\varphi$ is $L_\textnormal{loc}^1$ and satisfies
        \[
            \Delta \varphi = \delta_{z_0}.
        \]
        
        The last assertion follows by observing that one must necessarily have $g'(z_0) = \frac{1}{2}f''(z_0)$, meaning that
        \[
            \left(\frac{1}{2\pi} \log(g(\cdot))\right)'(z_0) = \frac{1}{2\pi} \frac{g'(z_0)}{g(z_0)} = \frac{1}{4\pi} \frac{f''(z_0)}{f'(z_0)},
        \]
        whence we deduce from the Cauchy-Riemann equations that
        \[
            \nabla h (z_0) = \frac{1}{4\pi} \overline{\left(\frac{f''(z_0)}{f'(z_0)}\right)}. \qedhere
        \]
    \end{proof}

    \begin{theorem}[Green's functions in $\R^2$, mixed]
        \label{thm:greens_functions_mixed}
        Suppose that $\Omega \subsetneq \R^2$ is a simply connected domain and that $z_0 \in \Omega$. Furthermore, assume that $\partial \Omega = \Gamma_D \sqcup \Gamma_N$, where $\Gamma_N$ is $C^1$ and open in $\partial \Omega$. Finally, suppose that $\map{f}{\Omega}{\mathbb{D} \setminus ((-1,-a]\times \{0\})}$, where $a > 0$, is a bijective conformal map of $\Omega$ onto the unit disk with a slit, satisfying $f(z_0)=0$ and extending continuously to the boundary. This map should send $\Gamma_D$ to the unit circle and $\Gamma_N$ to the interval $(-1,a] \times \{0\}$, and should extend analytically across $\Gamma_N$ (when viewed as a map on $\C$). Then the function $\map{\varphi}{\Omega}{\R}$ defined by
        \[
            \varphi(z) \coloneqq \frac{1}{2\pi} \log{(\abs{f(z)})}
        \]
        is in $L_\textnormal{loc}^1(\Omega)$, extends continuously to the boundary and satisfies
        \begin{align*}
            \Delta \varphi &= \delta_{z_0},\\
            \restr{\varphi}{\Gamma_D} &= 0,\\
            \restr{\partial_n \varphi}{\Gamma_N} &= 0,
        \end{align*}
        where $\partial_n$ denotes the normal derivative.
    \end{theorem}
    \begin{proof}
        The only change from \Cref{thm:greens_functions} is checking that the normal derivative vanishes on $\Gamma_N$. This follows by using conformality.
    \end{proof}
    

\end{document}